\documentclass[a4papet,10pt]{article}
\usepackage[italian, english]{babel}
\usepackage{amsfonts,amssymb,mathtools}
\usepackage{graphicx}
\usepackage{amsmath}
\usepackage[latin1]{inputenc} 
\usepackage{bbm}
\usepackage{pstricks}
\usepackage{psfrag}
\usepackage{appendix}
\usepackage{amsthm}
\usepackage{dsfont}
\usepackage{stackrel}
\usepackage{eufrak}
\usepackage{mathrsfs}
\usepackage{enumerate}
\usepackage{cite}
\usepackage{chngcntr}
\usepackage[a4paper,top=30mm,bottom=30mm,left=30mm,right=30mm]{geometry}
\usepackage{appendix}
\usepackage{authblk}
\usepackage{comment}

\counterwithin{figure}{section}
\numberwithin{equation}{section}

 \makeatletter
 \@addtoreset{equation}{section}
 \makeatother

 \newcounter{extralabel}[section]

 \newtheorem{ittheorem}{Theorem}
 \newtheorem{itlemma}{Lemma}
 \newtheorem{itproposition}{Proposition}
 \newtheorem{itdefinition}{Definition}
 \newtheorem{itcorollary}{Corollary}
 \newtheorem{itconjecture}{Conjecture}
 \newtheorem{itremark}{Remark}
 \newtheorem{itassumption}{Assumption}

 \newenvironment{theorem}{\addtocounter{extralabel}{1}
 \begin{ittheorem}}{\end{ittheorem}}

 \newenvironment{lemma}{\addtocounter{extralabel}{1}
 \begin{itlemma}}{\end{itlemma}}

 \newenvironment{proposition}{\addtocounter{extralabel}{1}
 \begin{itproposition}}{\end{itproposition}}

 \newenvironment{corollary}{\addtocounter{extralabel}{1}
 \begin{itcorollary}}{\end{itcorollary}}

 \newenvironment{remark}{\addtocounter{extralabel}{1}
 \begin{itremark}}{\end{itremark}}
 
 \newenvironment{assumption}{\addtocounter{extralabel}{1}
 \begin{itassumption}}{\end{itassumption}}

\newcommand{\prob}{\mathbb{P}}
\newcommand{\Ex}{\mathbb{E}}
\newcommand{\Rl}{\mathbb{R}}

\newcommand{\dom}{\mbox{dom}\,}
\newcommand{\rew}{\mathcal{X}}
\newcommand{\Brew}{\mathcal{B}(\mathcal{X})}
\newcommand{\pae}{\prob\mbox{-a.e.\ }\omega}
\newcommand{\dd}{\mathrm{d}}
\newcommand{\ee}{\mathrm{e}}

\newcommand{\N}{\mathbb{N}}
\newcommand{\bdot}{\boldsymbol{\cdot}}
\newcommand{\St}{\mathcal{C}}
\newcommand{\Sto}{\mathcal{C}\setminus\{c\}}
\newcommand{\FF}{\mathsf{F}}

\DeclareMathOperator*{\infd}{inf\vphantom{p}}

\DeclareMathOperator*{\esssup}{ess\,sup}

\newtheoremstyle{mystyle}
  {}
  {}
  {\itshape}
  {}
  {\bfseries}
  {.}
  { }
  {\thmname{#1}\thmnumber{ #2}\thmnote{ (#3)}}

\title{Quenched large deviations in renewal theory}
  
\author{
Frank den Hollander\footnote{Mathematisch Instituut, Universiteit Leiden, Niels Bohrweg 1, 2333 CA Leiden, The Netherlands,\\
\texttt{denholla@math.leidenuniv.nl}} 
\quad
Marco Zamparo\footnote{Dipartimento di Fisica, Universit\`a degli Studi di Bari e Istituto Nazionale di Fisica Nucleare,
  Sezione di Bari, via Amendola 173, 70126 Bari, Italy,\\
\texttt{marco.zamparo@uniba.it}}
}

\date{\today}

\begin{document}  

\maketitle                              

\begin{abstract}
  In this paper we introduce and study renewal-reward processes in random environments where each renewal involves a reward taking values in a Banach space. We derive quenched large deviation principles and identify the associated rate functions in terms of variational formulas involving correctors. We illustrate the theory with three examples: compound Poisson processes in random environments, pinning of polymers at interfaces with disorder, and returns of Markov chains in dynamic random environments. 

\\
{\it Acknowledgment.} FdH was supported by NWO Gravitation Grant 024.002.003-NETWORKS. MZ was supported by Apulia Region via the
project UNIBA044/REFIN - Research for Innovation.
\end{abstract}


\small
\tableofcontents
\normalsize


\section{Introduction}


\subsection{Background}

\emph{Random walks in random environments} \cite{zeitouni2006} have been in the focus of attention since the 1970's, exhibiting rich behaviour that is associated with slow-down phenomena and anomalous large deviations \cite{greven1994,dembo1996,gantert1998,zerner1998,pisztora1999_1,pisztora1999_2,comets2000,varadhan2003,dembo2004,rassoulagha2004,rosenbluth2006,peterson2009,yilmaz2009_1,yilmaz2010_1,yilmaz2010_2,rassoulagha2011,yilmaz2011_1,yilmaz2011_2,peterson2014,berger2018,bazaes2023_1,bazaes2023_2}. Large deviations have also been investigated for \emph{random polymers} and \emph{random walks in random potentials} \cite{rassoulagha2013,rassoulagha2014,georgiou2016,cometsbook,rassoulagha2017_1}, and more recently for \emph{random walks in dynamic random environments} \cite{yilmaz2009_2,balazs2019,rassoulagha2017_2,avena2010,avena2018}. Surprisingly, in spite of the great theoretical importance and the wide applicability of renewal processes, much less attention has been paid to \emph{renewal processes in random environments}. To the best of our knowledge, the only attempt in this direction was made in \cite{baxter1994}, in which the waiting times depend on a latent stochastic process and some of the standard limit theorems of renewal theory are generalised, including the renewal theorem and Blackwell's theorem. Viewed from a different perspective, renewal processes in random environments appear in the study of random polymers pinned at an interface with disorder \cite{giacominbook,denhollanderbook_2}, and of large deviations for words cut out from a random letter sequence \cite{birkner2010,denhollander2014}.

The purpose of the present paper is to develop \emph{large deviation principles} (LDPs) for renewal processes in random environments. Specifically, we assume that each renewal involves a reward taking a value in a real separable Banach space, and we characterise the quenched fluctuations of the total reward over time, i.e., fluctuations conditional on a typical realisation of the environment. In the absence of disorder, the total reward over time defines a so-called \textit{renewal-reward process} \cite{mitovbook}, so that in fact we deal with renewal-reward processes in random environments.  Our approach exploits subadditivity properties of renewal models, which in the absence of disorder lead to LDPs for the total reward under optimal hypotheses \cite{zamparo2021,zamparo2023}.

The problem we address fits naturally into the context of large deviations for random walks in random environments and random potentials, but it differs from the mainstream literature in two respects. First, the renewal processes we consider allow for renewal times with possibly unbounded and heavy-tailed increments, in contrast to the random walks in random environments and random potentials with bounded increments that are typically considered in the literature. Second, motivated by queuing theory \cite{asmussenbook}, insurance and finance \cite{rolskibook}, and statistical mechanics \cite{zamparo2019}, which consider rewards of various types, we investigate large deviations for random variables that span vector spaces, rather than process-level large deviations. LDPs for the total reward cannot be deduced from process-level LDPs, except for special cases where a contraction principle \cite{dembobook,denhollanderbook_1} can be employed and exponential moments are finite \cite{schied1998}.

In order to identify rate functions, we consider a class of models where the rate function in the quenched LDP can be related via convex conjugation to a cumulant generating function, which itself is traced back to a quenched free energy. Since the latter in general is known only as a subadditive limit, we also provide a variational formula for the quenched free energy, which involves correctors (gradients of sorts) and is deduced as a variational solution for the growth rate of a renewal equation in a random environment. This variational formula applies, in particular, to the free energy of the random pinning model and appears to be new in that context. Similar variational formulas have been obtained previously for random walks in random environments and random potentials \cite{rosenbluth2006,yilmaz2009_1,rassoulagha2013,georgiou2016,cometsbook,rassoulagha2017_1}.

The paper is organised as follows. Section \ref{sec:basicingredients} introduces the class of models considered, identified as generalised pinning models with general rewards, as well as some basic tools needed for the study of their large deviations, including Kingman's subadditive ergodic theorem. Section \ref{sec:mainresults} states our main assumptions and formulates our main results in the form of LDPs. Section \ref{sec:discussion} offers a brief discussion, while Section \ref{sec:examples} describes three examples to which the main results can be applied. Proofs are given in Sections \ref{proofs_constrainedmodel} and \ref{sec:freemodels}.


\subsection{Basic ingredients}
\label{sec:basicingredients}


\paragraph{Random environment.}

The random environment, also called \textit{disorder}, is sampled from a probability space $(\Omega,\mathcal{F},\prob)$ endowed with an ergodic measure-preserving transformation $f$. Expectation under $\prob$ is denoted by $\Ex$.


\paragraph{Renewals and rewards.}

Put $\N:=\{1,2,\ldots\}$, $\N_0:=\{0\}\cup\N$, and $\overline{\N}:=\N\cup\{\infty\}$. Let $p_\omega$ be a probability mass function on $\overline{\N}$ parametrised by $\omega\in\Omega$ in such a way that $\omega\mapsto p_\omega(s)$ is a measurable function for all $s$. In addition, let $(\rew,\|{\bdot}\|)$ be a real separable Banach space equipped with the Borel $\sigma$-field $\Brew$ and, for $s \in \overline{\N}$, let $\lambda_\omega(\bdot|s)$ be a probability measure on $\Brew$ parametrised by $\omega$ in such a way that $\omega \mapsto \lambda_\omega(A|s)$ is measurable for all $A\in\Brew$. We call $p_\omega$ the \textit{waiting-time distribution} and $\lambda_\omega(\bdot|s)$ the \textit{reward probability measure}. In fact, for a given disorder $\omega$, we consider a sequence $S_1,S_2,\ldots$ of \textit{waiting times} taking values in $\overline{\N}$ and a sequence $X_1,X_2,\ldots$ of \textit{rewards} taking values in $\rew$ whose joint law $P_\omega$ satisfies
\begin{equation*}
P_\omega\Big[S_1=s_1,\ldots,S_n=s_n,\,X_1\in A_1,\ldots,X_n=A_n\Big]
= \prod_{i=1}^n p_{f^{t_{i-1}}\omega}(s_i)\,\lambda_{f^{t_{i-1}}\omega}(A_i|s_i)
\end{equation*}
for $n\in \N$, $s_1,\ldots,s_n \in \overline{\N}$, and $A_1,\ldots,A_n \in \Brew$, where $t_0:=0$ and $t_i:=s_1+\cdots+s_i$ for $i \in \N$. This law describes some phenomenon that occurs at the \textit{renewal times} $T_i:=S_1+\cdots+S_i$, with $T_0:=0$, while an initial environment $\omega$ evolves over time by successive applications of the transformation $f$. The $i$th renewal involves the reward $X_i$, which can depend on $S_i$.  Expectation under $P_\omega$ is denoted by $E_\omega$.

The number of renewals up to time $t\in\N_0$ is the largest non-negative integer $N_t$ such that $T_{N_t}\le t$, and the \textit{total reward} up to time $t$ is
\begin{equation*}
W_t := \sum_{i=1}^{N_t} X_i
\end{equation*}
(empty sums are equal to zero). The process $\{W_t\}_{t\in\N_0}$ is a \textit{renewal-reward process in a random environment}, which turns out to be the classical renewal-reward process in the absence of disorder \cite{mitovbook}. We emphasise that $W_t$ is a random variable with respect to any probability space associated with the law of waiting times and rewards, because $\rew$ is separable \cite{ledouxbook}.


\paragraph{Pinning models.}

Let $v_\omega$ be a real function on $\overline{\N}$, which we call the \textit{potential}, parametrised by $\omega\in\Omega$ in such a way
that $\omega\mapsto v_\omega(s)$ is measurable for all $s$. For $t\in\N_0$, we define the \textit{pinning model} as the law $P_{\omega,t}$ determined by the Gibbs change of measure
\begin{equation*}
\frac{\dd P_{\omega,t}}{\dd P_\omega} := \frac{\ee^{H_{\omega,t}}}{E_\omega[\ee^{H_{\omega,t}}]},
\end{equation*}
where $H_{\omega,t} := \sum_{i=1}^{N_t}v_{f^{T_{i-1}}\omega}(S_i)$ is the \textit{Hamiltonian}. This definition generalises the standard pinning model \cite{giacominbook,denhollanderbook_2}, which will be recalled in Section \ref{sec:examples}, and contains as the special case $v_\omega=0$ the above renewal-reward system in a random environment.

Denote by $\mathcal{T}$ the random set $\{T_i\}_{i\in\N_0}$. Together with the pinning model, we consider the {\it constrained pinning model} corresponding to the law $Q_{\omega,t}$ obtained via the change of measure
\begin{equation*}
\frac{\dd Q_{\omega,t}}{\dd P_\omega} := \frac{\mathds{1}_{\{t\in\mathcal{T}\}}
\ee^{H_{\omega,t}}}{E_\omega[\mathds{1}_{\{t\in\mathcal{T}\}}\ee^{H_{\omega,t}}]}.
\end{equation*}
The constrained pinning model turns out to be a useful tool to investigate the pinning model. We generically refer to $P_{\omega,t}$ or $Q_{\omega,t}$ as pinning models.


\paragraph{Conditional independence.}

We aim to analyse the quenched large deviations of the renewal-reward process $\{W_t\}_{t\in\N_0}$, taking advantage of the fact that a renewal process at every renewal starts afresh in the current environment. This fact is formalised by the identity
\begin{align}
\nonumber
P_\omega&\Big[S_1=s_1,\ldots,S_{n+n'}=s_{n+n'},\,X_1\in A_1,\ldots,X_{n+n'}\in A_{n+n'}\,\Big|\,T_n=t\Big]\\
\nonumber
&=  P_\omega\Big[S_1=s_1,\ldots,S_n=s_n,\,X_1\in A_1,\ldots,X_n\in A_n\,\Big|\,T_n=t\Big]\\
&\qquad P_{f^t\omega}\Big[S_1=s_1,\ldots,S_{n'}=s_{n'},\,X_1\in A_1,\ldots,X_{n'}\in A_{n'}\Big]
\label{start_0}
\end{align}
for $n,n'\in \N$, $s_1,\ldots,s_{n+n'} \in \overline{\N}$, and $A_1,\ldots,A_{n+n'} \in \Brew$ (with the convention $\frac{0}{0}:=0$).


\paragraph{Subadditive ergodic theorem.}
 
To describe the asymptotic behaviour of probabilities, we use Kingman's subadditive ergodic theorem \cite{vianabook}:
\begin{itemize}
\item[$(\star)$] 
Let $\{F_t\}_{t \in \N}$ be a sequence of measurable functions on $\Omega$ such that $\Ex[\max\{0,F_1\}] < +\infty$ and $F_{t+t'} \le F_t + F_{t'}\circ f^t$ for all $t,t'\in \N$. Set $\mathcal{L} := \inf_{t \in \N}\{\Ex[\frac{F_t}{t}]\} < +\infty$. The following hold:\\ 
$(i)$ $\lim_{t\uparrow\infty} \Ex[\frac{F_t}{t}] = \mathcal{L}$.\\
$(ii)$ $\lim_{t\uparrow\infty} \frac{F_t(\omega)}{t} = \mathcal{L}$ $\pae$.
\end{itemize}
Note that the condition $\Ex[\max\{0,F_1\}] < +\infty$ can be replaced by the condition $\Ex[\max\{0,F_t\}] < +\infty$ for all $t \ge t_o$ for some $t_o\in \N$. In that case $\mathcal{L} := \inf_{t\ge t_o}\{\Ex[\frac{F_t}{t}]\}$.


\subsection{Main results}
 \label{sec:mainresults}

We investigate quenched large deviations for the process $\{W_t\}_{t\in\N_0}$ under the following assumptions about the waiting-time distribution $p_\omega$, the reward probability measures $\lambda_\omega(\bdot|s)$ for $s\in\N$, and the potential $v_\omega$ (which are assumed to be in force throughout the paper). Denote by $B_{w,\delta}$ the open ball of radius $\delta>0$ centred at $w\in\rew$. For $r\in\rew$ and $A\subseteq\rew$, let $r+A$ be the set of points $r+w$ with $w\in A$.

\begin{assumption}
\label{ass1}
{\rm There exist $m$ coprime integers $\sigma_1,\ldots,\sigma_m \in \N$ such that $\Ex[\log p_{\bdot}(\sigma_i)]>-\infty$ for $i=1,\ldots,m$.}  \hfill$\spadesuit$
\end{assumption}

\begin{assumption}
\label{ass2}
{\rm There exist a finite-dimensional subspace $\mathcal{V}\subseteq\rew$ and, for $s\in\N$, measurable functions $\omega\mapsto r_{\omega,s}$ valued in $\mathcal{V}$ and points $x_s\in\rew$ such that $\Ex[\log\lambda_{\bdot}(r_{\bdot,s}+B_{x_s,\delta}|s)]>-\infty$ for all $\delta>0$ and $\Ex[\sup_{s\in\N}\max\{0,\|r_{\bdot,s}\|-\eta s\}]<+\infty$ for some $\eta\ge 0$.  Moreover, for each $s\in\N$ there exists a compact set $K_s\subset\rew$ such that $\Ex[\log\lambda_{\bdot}(r_{\bdot,s}+K_s|s)]>-\infty$.}\hfill$\spadesuit$
\end{assumption}

\begin{assumption}
\label{ass3}
{\rm $\Ex[\min\{0,v_{\bdot}(s)\}]>-\infty$ for all $s\in\N$ and $\Ex[\sup_{s\in\N}\max\{0, v_{\bdot}(s)-\eta s\}]<+\infty$ for some real number $\eta\ge 0$.} \hfill$\spadesuit$
\end{assumption}

Our first main result concerns the large deviations of the total reward $W_t$ with respect to the constrained pinning model $Q_{\omega,t}$ for a typical disorder $\omega$. For convenience, we leave out normalisation and consider the random measure $\omega\mapsto\mu_{\omega,t}$ on $\Brew$ defined, for $t\in \N$, by
\begin{equation}
\mu_{\omega,t}:= E_\omega\bigg[\mathds{1}_{\big\{\frac{W_t}{t}\in \,\bdot\,,\,t\in\mathcal{T}\big\}} \ee^{H_{\omega,t}}\bigg].
\label{mudef}
\end{equation}
The following result about the cumulant generating function associated with $\mu_{\omega,t}$ comes before an LDP. Let $\rew^\star$ be the
topological dual of $\rew$, and let $\mathcal{R}_+$ be the set of positive random variables on $\Omega$ that are almost surely finite. For $\varphi\in\rew^\star$ and $\zeta\in\Rl$, consider the extended real number
\begin{equation}
\Upsilon_\varphi(\zeta):= \infd_{R\in\mathcal{R}_+}\prob\,\mbox{-}\esssup_{\omega\in\Omega} \bigg\{\log E_\omega\Big[\ee^{\varphi(X_1)+v_\omega(S_1)-\zeta S_1+R(f^{S_1}\omega)-R(\omega)}
\mathds{1}_{\{S_1<\infty\}}\Big]\bigg\}. 
\label{def:var_formula}
\end{equation}
Finally, denote by $z$ the function that maps $\varphi\in\rew^\star$ in $z(\varphi) :=\inf\{\zeta\in\Rl\colon\,\Upsilon_\varphi(\zeta)\le0\}$.

\begin{proposition}
\label{Zlim}
The following hold:\\
(i) $z$ is proper convex, with $z(0)$ finite, and is lower semi-continuous.\\
(ii) $\lim_{t\uparrow\infty}\frac{1}{t}\log\int_\rew \ee^{t\varphi(w)}\mu_{\omega,t}(\dd w)=z(\varphi)$ $\pae$ for every $\varphi\in\rew^\star$.
\end{proposition}

The variational definition of the function $z$, which in (\ref{def:var_formula}) involves the gradient $R\circ f^{S_1}-R$ of auxiliary random variables $R\in\mathcal{R}_+$ as a corrector, is our variational formula for the quenched free energy. The number $\zeta$ is reminiscent of a renewal equation, and accommodates for a Cram\'er-Lundberg parameter to describe its growth rate \cite{mitovbook}. We will comment further on this formula below. For now, we note that Proposition \ref{Zlim} suggests as a putative rate function the Legendre transform of $z$, which is the convex lower semi-continuous function $J$ that maps $w\in\rew$ in
\begin{equation*}
 J(w):=\sup_{\varphi\in\rew^\star}\big\{\varphi(w)-z(\varphi)\big\}.
\end{equation*}
The following theorem states a quenched LDP for the family of measures $\{\mu_{\omega,t}\}_{t \in \N}$.

\begin{theorem}
\label{WLDPc}
$\pae$ the family $\{\mu_{\omega,t}\}_{t \in \N}$ satisfies the weak LDP with rate function $J$, i.e.,\\
$(i)$ $\liminf_{t\uparrow\infty} \frac{1}{t}\log\mu_{\omega,t}(G) \ge-\inf_{w\in G} J(w)$ for all $G\subseteq\rew$ open.\\
$(ii)$ $\limsup_{t\uparrow\infty} \frac{1}{t}\log\mu_{\omega,t}(K) \le-\inf_{w\in K} J(w)$ for all $K\subset\rew$ compact.
\end{theorem}

It is desirable (for example, in order to be able to apply Varadhan's lemma \cite{dembobook,denhollanderbook_1}) that $\pae$ the family $\{\mu_{\omega,t}\}_{t \in \N}$ satisfies a full LDP with a good rate function. This means that the large deviation upper bound $(ii)$ in Theorem \ref{WLDPc} holds for all closed sets, and not only for compact sets, and that $J$ has compact level sets, i.e., the sets $\{w\in\rew:J(w)\le a\}$ are compact for all $a\in\Rl$. The following corollary of Theorem \ref{WLDPc} addresses this issue when the dimension of $\rew$ is finite.

\begin{corollary}
\label{FLDPc}
If $\rew$ is finite-dimensional and there exist real numbers $\xi>0$ and $M\ge 0$ such that
\begin{equation*}
\Ex\bigg[\sup_{s\in\N}\max\bigg\{0,\log\int_\rew\ee^{\xi\|x\|-Ms}\lambda_{\bdot}(\dd x|s)\bigg\}\bigg]<+\infty,
\end{equation*}
then $\pae$ the family $\{\mu_{\omega,t}\}_{t \in \N}$ satisfies the full LDP with good rate function $J$.
\end{corollary}

\begin{remark}
{\rm Since $\lim_{t\uparrow\infty}\frac{1}{t}\log\mu_{\omega,t}(\rew)=z(0)$ $\pae$ with $z(0)$ finite by Proposition \ref{Zlim}, Theorem \ref{WLDPc} and Corollary \ref{FLDPc} establish quenched LDPs with rate function $J+z(0)$ for the total reward with respect to the constrained pinning model.}
\end{remark}

Our second main result describes the large deviations of $W_t$ with respect to the pinning model $P_{\omega,t}$ by exploiting the large deviation bounds for the constrained model. As before, we leave normalisation aside and focus on the random measure $\omega\mapsto\nu_{\omega,t}$ on $\Brew$ defined, for $t\in \N$, by
\begin{equation}
\nu_{\omega,t}:= E_\omega\bigg[\mathds{1}_{\big\{\frac{W_t}{t}\in \,\bdot\,\big\}} \ee^{H_{\omega,t}}\bigg].
\label{nudef}
\end{equation}
Given an extended real number $\ell\in[-\infty,0]$, denote by $I_\ell$ the Legendre transform of $z_\ell:=\max\{z,\ell\}$, which associates $w\in\rew$ with
\begin{equation*}
I_\ell(w):=\sup_{\varphi\in\rew^\star}\big\{\varphi(w)-z_\ell(\varphi)\big\}.
\end{equation*}
Note that $z_\ell=z$ and $I_\ell=J$ if $\ell=-\infty$.

\begin{theorem}
\label{WLDPfree}
Suppose that there exists an $\ell\in[-\infty,0]$ such that $\pae$
\begin{equation}
\label{Ptailcond}
\begin{aligned}
&\adjustlimits\lim_{\epsilon\downarrow 0}\limsup_{s\uparrow\infty}\sup_{t\in\N_0}
\bigg\{\frac{1}{s}\log P_{f^t\omega}[S_1>s]-\epsilon \frac{t}{s}\bigg\}\le \ell,\\
&\adjustlimits\lim_{\epsilon\downarrow 0}\liminf_{s\uparrow\infty}\inf_{t\in\N_0}
\bigg\{\frac{1}{s}\log P_{f^t\omega}[S_1>s]+\epsilon \frac{t}{s}\bigg\}\ge\ell.
\end{aligned}
\end{equation}
Then the following hold: \\
$(i)$ $\lim_{t\uparrow\infty}\frac{1}{t}\log\int_\rew\ee^{t\varphi(w)}\nu_{\omega,t}(\dd w)=z_\ell(\varphi)$ $\pae$ for every $\varphi\in\rew^\star$.\\
$(ii)$ If either $\ell=-\infty$ or $\ell>-\infty$ and $\rew^\star$ is separable, then $\pae$ the family $\{\nu_{\omega,t}\}_{t \in \N}$ satisfies the weak LDP with rate function $I_\ell$.
\end{theorem}

The parameter $\ell$ in Theorem \ref{WLDPfree} plays the role of an \emph{exponential tail constant} for the waiting-time distribution.  When $\ell>-\infty$, we need separability of $\rew^\star$ to prove that the large deviation upper bound for compact sets holds uniformly $\pae$. The large deviation lower bound for open sets holds uniformly $\pae$ even without separability of $\rew^\star$. If $\rew$ is finite-dimensional, then the weak LDP can be easily promoted to a full LDP according to the following corollary of Theorem \ref{WLDPfree}.

\begin{corollary}
\label{FLDPfree}
Suppose that \eqref{Ptailcond} holds. Then, under the hypotheses of Corollary \ref{FLDPc}, $\pae$ the family $\{\nu_{\omega,t}\}_{t \in \N}$ satisfies the full LDP with good rate function $I_\ell$.
\end{corollary}

\begin{remark}
{\rm Since $\lim_{t\uparrow\infty}\frac{1}{t}\log\nu_{\omega,t}(\rew)=z_\ell(0)$ $\pae$ with $z_\ell(0)$ finite, Theorem \ref{WLDPfree} and Corollary \ref{FLDPfree} establish quenched LDPs with rate function $I_\ell+z_\ell(0)$ for the total reward with respect to the pinning model.}
\end{remark}


\subsection{Discussion}
\label{sec:discussion} 


\paragraph{Variational free energy.}

Assumption \ref{ass1} about the waiting-time distribution formulates a notion, suitable for the exponential scale of LDPs, of aperiodic probability mass function in a context with disorder.  Assumption \ref{ass3} on the potential is motivated by and applies to the polymer pinning model, discussed below in some detail, and its generalisations, such as the spatially extended pinning model \cite{caravenna2021}. We note that a homogeneous version, i.e., without disorder, of these assumptions is already present in \cite{zamparo2021}, where LDPs for the total reward are established within the framework of homogeneous pinning models.  According to $(ii)$ of Proposition \ref{Zlim} with $\varphi=0$, Assumptions \ref{ass1} and \ref{ass3} suffice to obtain the following limit, which shows that $z(0)$ is the quenched free energy of the (constrained)
pinning model:
\begin{equation*}
\lim_{t\uparrow\infty}\frac{1}{t}\log\mu_{\omega,t}(\rew)=\lim_{t\uparrow\infty}\frac{1}{t}\log E_\omega\big[\mathds{1}_{\{t\in\mathcal{T}\}}\ee^{H_{\omega,t}}\big]=z(0)~\pae.
\end{equation*}
We stress that the instance $\varphi=0$ puts the reward probability measures $\lambda_\omega(\bdot|s)$ out of the picture.  One of the services of Assumptions \ref{ass3} is to ensure that $z(0)$ is finite through the requirement $\Ex[\sup_{s\in\N}\max\{0, v_{\bdot}(s)-\eta s\}]<+\infty$ for some $\eta\ge 0$.

Our variational definition of the function $z$ based on (\ref{def:var_formula}) turns out to be a variational formula for the quenched free energy: $z(0)=\inf\{\zeta\in\Rl\colon\,\Upsilon_0(\zeta)\le0\}$ with
\begin{equation*}
\Upsilon_0(\zeta):= \infd_{R\in\mathcal{R}_+}\prob\,\mbox{-}\esssup_{\omega\in\Omega} \bigg\{\log E_\omega\Big[
\ee^{v_\omega(S_1)-\zeta S_1+R(f^{S_1}\omega)-R(\omega)}
\mathds{1}_{\{S_1<\infty\}}\Big]\bigg\}. 
\end{equation*}
It is interesting to note that the quantity $\Upsilon_0(\zeta)$ can itself be interpreted as the quenched free energy of a random walk. In fact, supposing for a moment that the waiting-time distribution $p_\omega$ has bounded support to make contact with the existing literature, $\Upsilon_0(\zeta)$ is the quenched free energy for the random walk $\{T_i\}_{i\in\N_0}$ of renewal times in a random environment when $v_\omega=0$ and $\zeta=0$ \cite{rosenbluth2006,yilmaz2009_1,rassoulagha2013}, or the quenched free energy for the random walk $\{T_i\}_{i\in\N_0}$ in the random potential $\omega\mapsto v_\omega-\zeta\mathrm{id}$ when $p_\omega$ is independent of $\omega$ \cite{rassoulagha2013,georgiou2016,cometsbook,rassoulagha2017_1}. These connections, together with the findings of \cite{rassoulagha2017_1}, suggest that the formula for $\Upsilon_0(\zeta)$ does not always have a minimiser over $\mathcal{R}_+$. Recasting this expression in terms of centered cocycles could lead to a minimiser, as shown for a class of random walks in random potentials \cite{georgiou2016,cometsbook,rassoulagha2017_1}. This issue, which we leave for future work, may open a new perspective in the study of the random pinning model.


\paragraph{Reward laws with disorder.}

Assumption \ref{ass2} on the distribution of rewards raises a new problem, namely, that of incorporating generic rewards into a renewal model with disorder. We are not aware of seeing this in similar or other contexts. The assumption is always verified by probability measures $\lambda_\omega(\bdot|s)$ that do not depend on the disorder $\omega$. In this case we can take $r_{\omega,s}=0$. In fact, for any probability measure $\lambda$ on $\Brew$ there exists a point $x\in\rew$ such that $\lambda(B_{x,\delta})>0$ for all $\delta>0$. On the contrary, if for every $y\in\rew$ it were possible to find a number $\delta_y>0$ such that $\lambda(B_{y,\delta_y})=0$, then the open covering $\{B_{y,\delta_y}\}_{y \in \rew}$ of $\rew$ would contain a countable subcollection covering $\rew$ by Lindel\"of's lemma, with the consequence that $\lambda(\rew) = 0$ instead of $\lambda(\rew) = 1$. Furthermore, there exists a compact set $K\subset\rew$ such that $\lambda(K)>0$, because $\lambda$ is tight as $\rew$ is separable (see \cite{bogachevbook}, Theorem 7.1.7).

In the presence of disorder, what Assumption \ref{ass2} requires is basically that the random probability measures $\omega\mapsto\lambda_\omega(\bdot|s)$ satisfy the properties of homogeneous laws, in mean and on the exponential scale of LDPs, under possibly a shift by the vectors $r_{\omega,s}$. The points $r_{\omega,s}$ are incorporated in the theory to account for a disorder-dependent deterministic component of rewards, which becomes dominant in the limiting case $\lambda_\omega(\bdot|s)=\delta_{r_{\omega,s}}$, with $\delta_x$ the Dirac measure centered at $x$. This limiting case allows, for instance, to study the large deviations of the Hamiltonian $H_{\omega,t}$, which is the total reward $W_t$ in a model where $\lambda_\omega(\bdot|s)=\delta_{v_\omega(s)}$. The vectors $r_{\omega,s}$ are required not to vary too much, in particular, to lie in a common finite-dimensional subspace $\mathcal{V}$. We make use of this assumption, together with the hypothesis $\Ex[\sup_{s\in\N}\max\{0,\|r_{\bdot,s}\|-\eta s\}]<+\infty$ for some $\eta\ge 0$, to ensure that the disorder-dependent deterministic component $\frac{1}{t}\sum_{i=1}^{N_t}r_{f^{T_{i-1}}\omega,S_i}$ of
the scaled total reward is in a compact set $\pae$. We stress that Assumption \ref{ass2} simplifies when the reward space $\rew$ is finite-dimensional, as in this case $\mathcal{V}$ can be taken equal to $\rew$ and the compact sets $K_s$ can be taken equal to the closure of an open ball.


\paragraph{Difficulties and open problems.}

Since Assumption \ref{ass2} allows for any homogeneous reward law, our LDPs contain as a particular case the LDPs established in \cite{zamparo2021} under optimal hypotheses for homogeneous pinning models. Our variational definition of the function $z$ generalises to a setting with disorder the variational definition given in \cite{zamparo2021} for the same function. We note that, in the homogeneous setting, the subadditivity properties of renewal models allow among others to extend the large deviation upper bounds from compact sets to open and closed convex sets under the same assumptions of the weak LDPs \cite{zamparo2021}. In the presence of disorder, our main assumptions and methods appear to be unable to reach the same result, because we exploit subadditivity properties after an approximation argument, as explained in Section \ref{sec:WLDPc}. We also note that, in the homogeneous setting, conditions for the full LDP to hold for infinite-dimensional rewards are known \cite{zamparo2023}. Such conditions come from large deviation theory of sums of i.i.d.\ random variables and are formulated in terms of exponential moments.  We leave the investigation of these conditions to models with disorder as an open problem.


\subsection{Examples}
\label{sec:examples}


\paragraph{Compound Poisson processes in random environments.}

Suppose that, in an economic scenario $\omega\in\Omega$, a customer can arrive in a shop with probability $\rho_\omega$, and can spend an amount of money smaller than or equal to $x\in\Rl$ with probability $F_\omega(x)$. If the economic scenario evolves according to an ergodic measure-preserving transformation $f$, then the customer arrives at time $s\in\N$ with probability
\begin{equation*}
p_\omega(s):=\prod_{t=1}^{s-1}(1-\rho_{f^t\omega})\rho_{f^s\omega}
\end{equation*}
(empty products are equal to one), and spends an amount of money with probability distribution $\lambda_\omega(\bdot|s)$ over $\mathcal{B}(\Rl)$ defined, for $x\in\Rl$, by
\begin{equation*}
\lambda_\omega((-\infty,x]|s):=F_{f^s\omega}(x).
\end{equation*}
Note that we are implicitly assuming that $F_\omega(x)$ is non-decreasing and right-continuous with respect to $x$. We also assume that $\omega\mapsto\rho_\omega$ and $\omega\mapsto F_\omega(x)$ are measurable. In this way, when other customers arrive progressively after the first customer, all of them are described by the law $P_\omega$ associated with the waiting time distribution $p_\omega$ and the reward probability measures $\lambda_\omega(\bdot|s)$ over $\rew:=\Rl$. The total reward $W_t$ turns out to be the amount of money earned by the shop up to time $t$.

We are able to characterise the large deviations of $W_t$ via our theory for a pinning model with potential $v_\omega$ equal to zero. Assumption \ref{ass1} is met if $\Ex[\log p_{\bdot}(1)]>-\infty$, i.e., $\Ex[\log\rho_{\bdot}]>-\infty$, while Assumption \ref{ass3} holds trivially.  Since the reward probability measures $\lambda_\omega(\bdot|s)$ depend on $s$ and $\omega$ only through $f^s\omega$, Assumption \ref{ass2} is satisfied if there exists a measurable function $\omega\mapsto r_\omega$ taking values in $\Rl$ such that $\Ex[|r_{\bdot}|]<+\infty$ and $\Ex[\log\{F_{\bdot}(r_{\bdot}+\delta)-F_{\bdot}(r_{\bdot}-\delta)\}]>-\infty$ for all $\delta>0$. The hypothesis of Corollary \ref{FLDPc} becomes
$\Ex[\log\int_\Rl\ee^{\xi|x|}\dd F_{\bdot}(x)]<+\infty$ for some real number $\xi>0$. Thus, Theorem \ref{WLDPfree} and Corollary \ref{FLDPfree} give the following result after we show that the condition (\ref{Ptailcond}) is verified with $\ell=\Ex[\log(1-\rho_{\bdot})]$.

\begin{corollary}
\label{Poisson}
Suppose that $\Ex[\log\rho_{\bdot}]>-\infty$ and that there exists a real measurable function $\omega\mapsto r_\omega$ such that $\Ex[|r_{\bdot}|]<+\infty$ and
$\Ex[\log\{F_{\bdot}(r_{\bdot}+\delta)-F_{\bdot}(r_{\bdot}-\delta)\}]>-\infty$ for all $\delta>0$. Then the following hold:\\
$(i)$ $\pae$ the family $\{\nu_{\omega,t}\}_{t\in \N}$, associated with the total profit, satisfies the weak LDP with rate function $I_\ell$, where $\ell:=\Ex[\log(1-\rho_{\bdot})]\in[-\infty,0]$.\\
$(ii)$ If, moreover, $\Ex[\log\int_\Rl\ee^{\xi|x|}\dd F_{\bdot}(x)]<+\infty$ for some number $\xi>0$, then $\pae$ the family $\{\nu_{\omega,t}\}_{t\in \N}$ satisfies the full LDP with
good rate function $I_\ell$.
\end{corollary}

\begin{proof}[Proof of Corollary \ref{Poisson}]
It remains to verify that the condition (\ref{Ptailcond}) holds with $\ell:=\Ex[\log(1-\rho_{\bdot})]$. The waiting-time tail probability is now, for $s\in\N_0$,
\begin{equation*}
P_\omega[S_1>s]=\prod_{i=1}^s(1-\rho_{f^i\omega}).
\end{equation*}

First we prove that (\ref{Ptailcond}) is fulfilled with the above $\ell$ when $\Ex[\log(1-\rho_{\bdot})]>-\infty$. To this aim, we note that Birkhoff's ergodic theorem allows us to find a set $\Omega_o\in\mathcal{F}$ with $\prob[\Omega_o]=1$ such that, for $\omega\in\Omega_o$, $\lim_{s\uparrow\infty}\frac{1}{s}\log P_\omega[S_1>s]=\ell$. Thus, for a given $\omega\in\Omega_o$ and for each $\epsilon>0$, there exists a constant $k$ such that the bounds $-k+(\ell-\epsilon/2)s\le \log P_\omega[S_1>s]\le k+(\ell+\epsilon/2)s$ are valid for all $s\in\N_0$. These bounds imply
\begin{align}
\nonumber
\log P_{f^t\omega}[S_1>s]&= \log P_\omega[S_1>t+s]-\log P_\omega[S_1>t]\\
\nonumber
&\le 2k+(\ell+\epsilon/2)(t+s)-(\ell-\epsilon/2)t\le 2k+(\ell+\epsilon) s+\epsilon t,
\end{align}
which shows that the first line of \eqref{Ptailcond} holds. A similar argument proves the second line.

In the case $\ell:=\Ex[\log(1-\rho_{\bdot})]=-\infty$, the second line of \eqref{Ptailcond} is trivial. Regarding the first, put $\rho_\omega^k:=\min\{\rho_\omega,1-1/k\}$ for $k\in\N$. Since $P_\omega[S_1>s]\le\prod_{i=1}^s(1-\rho_{f^i\omega}^k)$ for all $s$ and $k$, and $\Ex[\log(1-\rho_{\bdot}^k)]>-\infty$, we can conclude as before that there exists $\Omega_o\in\mathcal{F}$ with $\prob[\Omega_o]=1$ such that
\begin{equation*}
\adjustlimits\lim_{\epsilon\downarrow 0}\limsup_{s\uparrow\infty}\sup_{t\in\N_0}
\bigg\{\frac{1}{s}\log P_{f^t\omega}[S_1>s]-\epsilon \frac{t}{s}\bigg\}\le \Ex\big[\log(1-\rho_{\bdot}^k)\big]
\end{equation*}
for all $\omega\in\Omega_o$ and $k\in\N$. Letting $k\uparrow\infty$, we see the monotone convergence theorem demonstrates the first line of \eqref{Ptailcond}.
\end{proof}


\paragraph{Pinning of polymers at interfaces with disorder.}

We recover the standard pinning model \cite{giacominbook,denhollanderbook_2} as follows.  Suppose that the disorder $\omega:=\{\omega_\tau\}_{\tau\in\N_0}$ is a real sequence, which is sampled from a probability space $(\Omega,\mathcal{F},\prob)$ in such a way that the canonical projections $\omega\mapsto\omega_\tau$ form a sequence of i.i.d.\ random variables, and that the transformation $f$ is the left-shift acting on $\omega$. Next, suppose that the waiting-time distribution is independent of $\omega$ and is defined, for $s\in\N$, by
\begin{equation*}
p_\omega(s)=p(s):=\frac{L(s)}{s^{\alpha+1}}
\end{equation*}
with an index $\alpha\ge 0$ and a slowly varying function at infinity $L$. Finally, given parameters $h,\beta\in\Rl$, consider the potential $v_\omega(s):=h+\beta\omega_s$ for $s\in\N$, whose associated Hamiltonian turns out to be
\begin{equation*}
H_{\omega,t}=\sum_{i=1}^{N_t}(h+\beta\omega_{T_i}).
\end{equation*}
This model represents a heteropolymer consisting of $t$ monomers that interacts with a substrate through the monomers $T_1,\ldots,T_{N_t}$. Monomer $\tau$ has binding energy $h+\beta\omega_\tau$, and the spacing along the polymer chain of pinned units has a heavy-tailed distribution. If the last monomer is forced to be bound to the substrate, then the polymer is described by the constrained Gibbs measure $Q_{\omega,t}$, otherwise the law is the Gibbs measure $P_{\omega,t}$.

Our theory allows us to study the fluctuations of the number $N_t$ of pinned monomers. In fact, we have $W_t=N_t$ $P_\omega$-a.s.\ when $\rew:=\Rl$ and $\lambda_\omega(\bdot|s):=\delta_1$ for each $s\in\N$. Our theory also allows us to investigate the fluctuations of the polymer excursions between pinned units. To this aim, we may suppose that, e.g., $\rew$ is the Hilbert space of square-summable real functions on $\N$ endowed with the Euclidean inner product $\langle\bdot,{\bdot}\rangle$ and the corresponding norm $\|{\bdot}\|$. The space $\rew$ is separable, and so is its topological dual $\rew^\star$. Then, given an orthonormal basis $\{e_s\}_{s\in\N}$ of $\rew$, we can put $\lambda_\omega(\bdot|s):=\delta_{e_s}$ for $s\in\N$. In this framework, we have $W_t=\sum_{i=1}^{N_t}e_{S_i}=\sum_{s\in\N}\#_t(s)e_s$ $P_\omega$-a.s., where $\#_t(s)$ is a non-negative integer random variable that counts the number of times that a waiting time of size $s$ has occurred up to time $t$.

Both for the number of pinned monomers and the number of polymer excursions between pinned units, the reward probability measures do not depend on $\omega$, so that Assumption \ref{ass2} is automatically satisfied, as seen in Section \ref{sec:discussion}. Assumption \ref{ass1} is trivial because $p(s)>0$ for all sufficiently large $s$ by definition, and the condition in \eqref{Ptailcond} is met with $\ell:=0$.  Assumption \ref{ass3} requires that $\int_\Omega|\omega_0|\,\prob[\dd \omega]<+\infty$. Thus, Theorems \ref{WLDPc} and \ref{WLDPfree} provide quenched weak LDPs with respect to the constrained and non-constrained model, respectively. Regarding the number of pinned monomers, the LDP is actually full and the rate function is good according to Corollaries \ref{FLDPc} and \ref{FLDPfree}. We note, in general, that if $\int_\Omega|\omega_0|\,\prob[\dd \omega]<+\infty$ and rewards are bounded by a common constant $\rho>0$, i.e., $\lambda_\omega(B_{0,\rho}|s)=1$ for all $s\in\N$ and $\omega\in\Omega$ as in the above cases, then $z(\varphi)\ge0$ for all $\varphi\in\rew^\star$, so that $z_\ell=z$ and $I_\ell=J$. This follows from $(ii)$ of Proposition \ref{Zlim} since, for $\varphi\in\rew^\star$ and $t\in\N$, we have
\begin{align}
\nonumber
\int_\rew\ee^{t\varphi(w)}\mu_{\omega,t}(\dd w)&= E_\omega\Big[\mathds{1}_{\{t\in\mathcal{T}\}}\ee^{\varphi(W_t)+H_{\omega,t}}\Big]\\
\nonumber
&\ge E_\omega\Big[\mathds{1}_{\{T_1=t\}}\ee^{\varphi(W_t)+H_{\omega,t}}\Big]\\
\nonumber
&=p(t)\!\int_\rew\ee^{\varphi(x)+h+\beta\omega_t}\lambda_\omega(\dd x|t)\ge
p(t)\,\ee^{-\|\varphi\|\rho+h+\beta\omega_t}.
\end{align}
On the other hand, the hypothesis $\int_\Omega|\omega_0|\,\prob[\dd \omega]<+\infty$ implies $\lim_{t\uparrow\infty}\frac{\omega_t}{t}=0$ $\pae$, which can be easily justified by means of the strong law of large numbers. In conclusion, Theorems \ref{WLDPc} and \ref{WLDPfree} lead to the following results.

\begin{corollary}
Suppose that $\int_\Omega|\omega_0|\,\prob[\dd \omega]<+\infty$. Then the following hold:\\
$(i)$ $\pae$ the families $\{\mu_{\omega,t}\}_{t\in \N}$ and $\{\nu_{\omega,t}\}_{t\in \N}$, associated with the number of pinned monomers, satisfy the full LDP with good rate
function $J$.\\
$(ii)$ $\pae$ the families $\{\mu_{\omega,t}\}_{t\in \N}$ and $\{\nu_{\omega,t}\}_{t\in \N}$, associated with the polymer excursions between pinned units, satisfy the weak LDP with rate function $J$.
\end{corollary}

The example of the number of pinned units allows to appreciate the effect of disorder on rate functions. The following proposition makes use of a known smoothing effect of disorder in the pinning model (under more restrictive hypotheses on disorder than necessary for the sake of simplicity) to unveil some properties of the rate function $J$ of the number of pinned monomers. Note that, as $\frac{N_t}{t}\in[0,1]$ for all $t$, the large deviation lower bound for open sets implies $J(w)=+\infty$ for $w\notin[0,1]$. Put $u:=\frac{\sum_{s\in\N}p(s)}{\sum_{s\in\N}sp(s)}\in(0,1)$ if $\sum_{s\in\N}sp(s)<+\infty$ and $u:=0$ otherwise.

\begin{proposition}
\label{disorder_effect}
Assume that $p(s)>0$ for all $s\in\N$ and that either $\omega_0$ is bounded with full probability or is Gaussian distributed. Then the following hold:\\
$(i)$ If $\beta=0$ and $u>0$, then the rate function $J$ of the number of pinned monomers has an affine stretch on $(0,u]$, whereas it is strictly convex and infinitely differentiable on $(u,1)$.\\
$(ii)$ If either $\beta=0$ and $u=0$ or $\beta\ne 0$, then $J$ is strictly convex and infinitely differentiable on $(0,1)$.
\end{proposition}

Thus, for a model with $u>0$, the rate function $J$ of the number of pinned monomers has an affine stretch terminating at the point $u$ when disorder is absent, i.e., when $\beta=0$. But as soon as disorder comes into play, i.e., when $\beta\ne 0$, the rate function $J$ loses such an affine stretch. We note that the phenomenon of affine stretches in rate functions of homogeneous pinning models have been investigated in detail in \cite{zamparo2019,zamparo2021_1}.

\begin{proof}[Proof of Proposition \ref{disorder_effect}]
Let us denote the free energy of the constrained pinning model $z(0)$ by $\FF(h)$ to highlight the dependence on the parameter $h$.  Since we are considering a problem where $W_t=N_t$ $P_\omega$-a.s., so that $\int_\rew \ee^{t\varphi(w)}\mu_{\omega,t}(\dd w)=E_\omega[\mathds{1}_{\{t\in\mathcal{T}\}}\ee^{\varphi(W_t)+H_{\omega,t}}]=E_\omega[\mathds{1}_{\{t\in\mathcal{T}\}}\ee^{\varphi(1)N_t+H_{\omega,t}}]$, $(ii)$ of Proposition \ref{Zlim} shows that $z(\varphi)=\FF(h+k)$ with $k:=\varphi(1)$. It therefore follows that, for $w\in\Rl$,
\begin{equation}
J(w):=\sup_{\varphi\in\rew^\star}\big\{\varphi(w)-z(\varphi)\big\}=\sup_{k\in\Rl}\big\{wk-\FF(h+k)\big\}=\sup_{k\in\Rl}\big\{wk-\FF(k)\big\}-wh.
\label{J_free_energy_pinning}
\end{equation}

The overall features of the free energy under the hypotheses of the proposition on disorder have been characterised and are now needed. In order to make contact with the existing literature, we suppose without loss of generality that $\int_\Omega\omega_0\prob[\dd\omega]=0$ and $\int_\Omega\omega_0^2\prob[\dd\omega]=1$. To begin with, we recall that there exists a number $h_c$ such that $\FF(h)=0$ for $h\le h_c$ and $\FF(h)>0$ for $h>h_c$ (see \cite{giacominbook}, Chapter 5). The function that maps $h$ in $\FF(h)$ is convex on the real line and is infinitely differentiable and strictly convex on the open interval $(h_c,+\infty)$ (see \cite{giacomin2006_1}, Theorem 2.1 and \cite{giacomin2020}, Theorem B.1). In the absence of disorder, i.e., when $\beta=0$, we have that $h_c=-\log\sum_{s\in\N}p(s)$ and that $\FF(h)$ for $h>h_c$ is the unique positive real number that solves the equation $\sum_{s\in\N}p(s)\ee^{-\FF(h) s}=\ee^{-h}$ (see \cite{giacominbook}, Proposition 1.1). Then, when $\beta=0$, we can easily verify that $\lim_{h\downarrow h_c}\FF'(h)=u$. This limit is affected by disorder, which has a smoothing effect. In fact, when $\beta\ne 0$, it is known that $\lim_{h\downarrow h_c}\FF'(h)=0$ (see \cite{giacomin2006_2}, Theorem 2.1). In all cases, we find $\lim_{h\uparrow+\infty}\FF'(h)=1$ since $p(1)>0$. To prove this claim, we note that, for $h>0$, $t\in\N$, and $\omega\in\Omega$,
\begin{align}
\nonumber
\ee^{\sum_{i=1}^t(h+\beta\omega_i)+t\log p(1)}= E_\omega\Big[\mathds{1}_{\{S_1=1,\ldots,S_t=1\}}\ee^{H_{\omega,t}}\Big]
&\le  E_\omega\Big[\mathds{1}_{\{t\in\mathcal{T}\}}\ee^{H_{\omega,t}}\Big]\\
\nonumber
&\le \ee^{ht}E_\omega\Big[\mathds{1}_{\{t\in\mathcal{T}\}}\ee^{H_{\omega,t}-hN_t}\Big],
\end{align}
so that $h+\log p(1)\le \FF(h)\le h+F(0)$ by Proposition \ref{Zlim} and Birkhoff's ergodic theorem. The latter bounds imply $\lim_{h\uparrow+\infty}\frac{F(h)}{h}=1$. In this way, since convexity of $\FF$ entails $\frac{\FF(h)-\FF(0)}{h}\le\FF'(h)\le\frac{\FF(k)-\FF(h)}{k-h}$ for $\max\{0,h_c\}<h<k$, we get $\lim_{h\uparrow+\infty}\FF'(h)=1$ from here by letting $k\uparrow+\infty$ first and then $h\uparrow+\infty$.

Coming back to formula (\ref{J_free_energy_pinning}), the listed overall features of the free energy allow us to conclude that
\begin{equation*}
J(w)=\begin{cases}
w(h_c-h) & \mbox{if $\beta=0$ and $w\in[0,u]$}\\
w(k-h)-\FF(k)  & \mbox{if $\beta=0$ and $w\in(u,1)$ or $\beta\ne 0$ and $w\in(0,1)$}
\end{cases},
\end{equation*}
where $k>h_c$ is the unique real number that satisfies $\FF'(k)=w$. The implicit function theorem demonstrates that $J$ is strictly convex and infinitely differentiable on the interval $(u,1)$ for $\beta=0$ and on the interval $(0,1)$ for $\beta\ne0$.
\end{proof}


\paragraph{Returns of Markov chains in dynamic random environments.}

Let $\St$ be a finite set and let $\omega\mapsto K_\omega:=\{K_\omega(a,b)\}_{a,b\in\St}$ be a random stochastic matrix over $\St$, i.e., a random non-negative matrix that satisfies $\sum_{b\in\St}K_\omega(a,b)=1$ for all $a$ and $\omega$. If the environment $\omega$ evolves by successive applications of an ergodic measure-preserving transformation $f$, then the random matrix $\omega\mapsto K_\omega$ defines a Markov chain in a dynamic random environment, which at time $t$ jumps from state $a$ to state $b$ with probability $K_{f^t\omega}(a,b)$.

Given a distinguished state $c\in\St$, the returns of the Markov chain to the state $c$ define a renewal process in a random environment. Its waiting-time distribution reads, for $s\in\N$,
\begin{equation*}
p_\omega(s):=
\begin{cases}
K_\omega(a_0,a_s) & \mbox{if }s=1\\
\sum_{a_1\in\Sto}\cdots \sum_{a_{s-1}\in\Sto}\prod_{i=0}^{s-1}K_{f^i\omega}(a_i,a_{i+1})& \mbox{if }s\ge 2
\end{cases}
\end{equation*}
with $a_0=a_s:=c$. We can use our theory to investigate, e.g., the number of distinct states that the Markov chain explores during its excursions from the state $c$. To this aim, for each $s\ge 2$ such that $p_\omega(s)>0$, denoting by $\#\{a_1,\ldots,a_s\}$ the number of distinct elements in the collection $\{a_1,\ldots,a_s\}$, we introduce over $\rew:=\Rl$ the reward probability measure
\begin{equation*}
\lambda_\omega(\bdot|s):=\frac{1}{p_\omega(s)}\sum_{a_1\in\Sto}\cdots \sum_{a_{s-1}\in\Sto}
\delta_{\#\{a_1,\ldots,a_s\}}\prod_{i=0}^{s-1}K_{f^i\omega}(a_i,a_{i+1})
\end{equation*}
with $a_0=a_s:=c$. We put $\lambda_\omega(\bdot|1):=\delta_1$ if $p_\omega(1)>0$ and $\lambda_\omega(\bdot|s):=\delta_0$ whenever $p_\omega(s)=0$. Under the law $P_\omega$ associated with $p_\omega$ and $\lambda_\omega(\bdot|s)$, the total reward $W_t$ is almost surely the number of visited states in all excursions from the state $c$ up to time $t$.

Assumption \ref{ass1} immediately holds in the easy case $\Ex[\log K_{\bdot}(a,b)]>-\infty$ for all $a,b\in\St$, which gives $\Ex[\log p_{\bdot}(s)]>-\infty$ for every $s$.  Assumption \ref{ass3} is trivial, since there is no potential in this example. Assumption \ref{ass2} is verified as follows. Fix $s\in\N$ and, for $n\in\{0,\ldots,s\}$, denote by $\Omega_n$ the set of all $\omega\in\Omega$ with the property that $n$ is the smallest non-negative integer $i$ such that $\lambda_\omega(\{i\}|s)\ge \frac{1}{s+1}$. The sets $\Omega_0,\ldots,\Omega_s$ are measurable, because they inherit this property from the random stochastic matrix $\omega\mapsto K_\omega$, and are disjoint. Moreover, $\cup_{n=0}^s\Omega_n=\Omega$ since $\sum_{n=0}^s\lambda(\{n\}|s)=1$. With such sets, consider the measurable function $\omega\mapsto r_{\omega,s}$ that takes value $n$ over $\Omega_n$. This function leads to the fulfilment of Assumption \ref{ass2}, since $|r_{\omega,s}|\le s$ and $\lambda_\omega(\{r_{\omega,s}\}|s)\ge \frac{1}{s+1}$ for all $\omega$ by construction. Finally, we note that the hypothesis of Corollary \ref{FLDPc} is satisfied with, e.g., $\xi=1$ and $M=1$ since $\lambda_\omega([0,s]|s)=1$ for all $s$ and $\omega$. Thus, Theorem \ref{WLDPfree} and Corollary \ref{FLDPfree} give the following result after we show that the condition (\ref{Ptailcond}) holds with some $\ell$.

\begin{corollary}
\label{Markov_chain}
Assume that $\Ex[\log K_{\bdot}(a,b)]>-\infty$ for all $a,b\in\St$. Then there exists a finite number $\ell\le 0$ such that $\pae$ the family $\{\nu_{\omega,t}\}_{t\in \N}$, associated with the total number of visited states in the excursions from the distinguishable state $c$, satisfies the full LDP with good rate function $I_\ell$.
\end{corollary}

\begin{proof}[Proof of Corollary \ref{Markov_chain}]
It remains to verify that condition (\ref{Ptailcond}) holds with a finite number $\ell$.  The waiting-time tail probability reads, for $s\in\N$,
\begin{equation}
P_\omega[S_1>s]= \sum_{a_1\in\Sto}\cdots \sum_{a_s\in\Sto}K_\omega(a_0,a_1)\cdots K_{f^{s-1}\omega}(a_{s-1},a_s)
\label{tail_Markov}
\end{equation}
with $a_0:=c$. 

Denoting by $M_\omega$ the restriction of $K_\omega$ to $\Sto$ and putting $M_\omega^t:=M_{f\omega}\cdots M_{f^t\omega}$ for $t\in\N$ with $M_\omega^0$ the identity matrix, we can recast (\ref{tail_Markov}) as
\begin{equation*}
P_\omega[S_1>s]=\sum_{a\in\Sto}\sum_{b\in\Sto} K_\omega(c,a)M^{s-1}_\omega(a,b).
\end{equation*}
Since $\Ex[\log M_{\bdot}(a,b)]=\Ex[\log K_{\bdot}(a,b)]>-\infty$ for all $a,b\in\Sto$, there exist a finite number $\ell$ and a set $\Omega_o\in\mathcal{F}$ with $\prob[\Omega_o]=1$ such that $\lim_{t\uparrow\infty}\frac{1}{t}\log M_\omega^t(a,b)=\ell$ for all $a,b\in\Sto$ and $\omega\in\Omega_o$ (see \cite{kingman1973}, Theorem 5). We can choose $\Omega_o$ in order to also have $\lim_{t\uparrow\infty}\frac{1}{t}\log K_{f^t\omega}(c,a)=0$ for all $a\in\Sto$ and $\omega\in\Omega_o$. A simple way to justify this claim involves Birkhoff's ergodic theorem. Let us show that (\ref{Ptailcond}) holds for all $\omega\in\Omega_o$ with such $\ell$.

Pick $\omega\in\Omega_o$ and $\epsilon>0$. By construction, there exists a constant $k$ such that $(\ell-\epsilon/3)t-k\le\log M_\omega^t(a,b)\le (\ell+\epsilon/3)t+k$ and $\log
K_{f^t\omega}(c,a)\ge -(\epsilon/3) t-k$ for all $a,b\in\Sto$ and $t\in\N$. The latter can be extended to also include the case $t=0$. In this way, given any $a_o\in\Sto$, the identity
\begin{equation*}
\sum_{a\in\Sto}M^t_\omega(a_o,a)M^{s-1}_{f^t\omega}(a,b)=M^{t+s-1}_\omega(a_o,b)
\end{equation*}
shows that, for $s\in\N$, $t\in\N_0$, and $b\in\Sto$,
\begin{equation*} 
\ee^{(\ell-\epsilon)(s-1)-(2\epsilon/3) t-2k}\le\sum_{a\in\Sto}M^{s-1}_{f^t\omega}(a,b)\le |\St|\ee^{(\ell+\epsilon)s+\epsilon t+2k}.
\end{equation*}
We also have, for $t\in\N_0$ and $a\in\Sto$,
\begin{equation*}
\ee^{ -(\epsilon/3) t-k}\le K_{f^t\omega}(c,a)\le 1.
\end{equation*}
We note that
\begin{equation*}
P_{f^t\omega}[S_1>s]=\sum_{a\in\Sto}\sum_{b\in\Sto} K_{f^t\omega}(c,a)M^{s-1}_{f^t\omega}(a,b)\le |\St|^2\ee^{(\ell+\epsilon)s+\epsilon t+2k}
\end{equation*}
and
\begin{equation*}
P_{f^t\omega}[S_1>s]\ge  \ee^{(\ell-\epsilon)(s-1)-\epsilon t-3k}
\end{equation*}
for all $s\in\N$ and $t\in\N_0$. These bounds show that
\begin{equation*}
\limsup_{s\uparrow\infty}\sup_{t\in\N_0}\bigg\{\frac{1}{s}\log P_{f^t\omega}[S_1>s]-\epsilon \frac{t}{s}\bigg\}\le \ell+\epsilon
\end{equation*}
and
\begin{equation*}
\liminf_{s\uparrow\infty}\inf_{t\in\N_0}\bigg\{\frac{1}{s}\log P_{f^t\omega}[S_1>s]+\epsilon \frac{t}{s}\bigg\}\ge \ell-\epsilon.
\end{equation*}
The arbitrariness of $\epsilon$ proves (\ref{Ptailcond}).
\end{proof}


\section{Constrained LDPs}
\label{proofs_constrainedmodel}


In this section we prove Proposition \ref{Zlim}, Theorem \ref{WLDPc}, and Corollary \ref{FLDPc}. Section \ref{sec:super} introduces some fundamental limits by combining supermultiplicativity properties of the constrained pinning model with Kingman's subadditive ergodic theorem. Section \ref{sec:ren_eq} verifies Proposition \ref{Zlim}. In Section \ref{sec:WLDPc} we prove the quenched weak LDP for the family of measures $\{\mu_{\omega,t}\}_{t\in\N}$, identifying the rate function in Section \ref{sec:Legendre}. This demonstrates Theorem \ref{WLDPc}. Finally, in Section \ref{sec:FLDPc} we obtain the quenched full LDP of Corollary \ref{FLDPc}.


\subsection{Supermultiplicativity}
\label{sec:super}


\paragraph{Supermultiplicativity.}

The identity in (\ref{start_0}) leads to a supermultiplicativity property for the random measures $\omega\mapsto\mu_{\omega,t}$ defined by (\ref{mudef}), as shown in the following lemma.  For $A,A'\subseteq\rew$ and $\alpha\in\Rl$, define $\alpha A:=\{\alpha w:w\in A\}$ and $A+A':=\{w+w':w\in A,\, w'\in A'\}$.
\begin{lemma}
\label{super_mu}
For $t,t'\in \N$ and $A,A'\in\Brew$, $\mu_{\omega,t+t'}(\alpha A+\alpha' A') \ge\mu_{\omega,t}(A)\mu_{f^t\omega,t'}(A')$ with $\alpha:=\frac{t}{t+t'}$ and $\alpha':=\frac{t'}{t+t'}$. In particular, $\mu_{\omega,t+t'}(A) \ge\mu_{\omega,t}(A)\mu_{f^t\omega,t'}(A)$ if $A$ is convex.
\end{lemma}

\begin{proof}[Proof of Lemma \ref{super_mu}]
Writing $\frac{W_{t+t'}}{t+t'} = \alpha \frac{W_t}{t}+\alpha'\frac{W_{t+t'}-W_t}{t'}$, we see that $\frac{W_{t+t'}}{t+t'}\in \alpha A+\alpha' A'$ whenever $\frac{W_t}{t} \in A$ and $\frac{W_{t+t'}-W_t}{t'}\in A'$. It follows that
\begin{align}
\nonumber
\mu_{\omega,t+t'}(\alpha A+\alpha' A')
&:=E_\omega\bigg[\mathds{1}_{\big\{\frac{W_{t+t'}}{t+t'} \in \alpha A+\alpha' A',\,t+t'\in\mathcal{T} \big\}}\ee^{H_{\omega,t+t'}}\bigg]\\
\nonumber
&\ge E_\omega\bigg[\mathds{1}_{\big\{\frac{W_t}{t} \in A,\,\frac{W_{t+t'}-W_t}{t'}\in A',\,t+t'\in\mathcal{T}\big\}}\ee^{H_{\omega,t+t'}}\bigg]\\
\nonumber
&\ge E_\omega\bigg[\mathds{1}_{\big\{\frac{W_t}{t} \in A,\,t\in\mathcal{T},\,
\frac{W_{t+t'}-W_t}{t'}\in A',\,t+t'\in\mathcal{T}\big\}}\ee^{H_{\omega,t+t'}}\bigg],
\end{align}
where the last bound accounts for the constraint $t\in\mathcal{T}$. This constraint gives
\begin{align}
\nonumber
E_\omega
&\bigg[\mathds{1}_{\big\{\frac{W_t}{t} \in A,\,t\in\mathcal{T},\,\frac{W_{t+t'}-W_t}{t'}\in A',\,t+t'\in\mathcal{T}\big\}}\ee^{H_{\omega,t+t'}}\bigg]\\
\nonumber
&=\sum_{n\in\N}\sum_{n'\in\N}E_\omega\bigg[\mathds{1}_{\big\{\frac{W_t}{t} \in A,\,T_n=t \big\}}\ee^{H_{\omega,t}}
\mathds{1}_{\big\{\frac{W_{t+t'}-W_t}{t'}\in A',\,T_{n+n'}=t+t'\big\}}\ee^{H_{\omega,t+t'}-H_{\omega,t}}\bigg]\\
\nonumber
&= E_\omega\bigg[\mathds{1}_{\big\{\frac{W_t}{t} \in A,\,t\in\mathcal{T} \big\}}\ee^{H_{\omega,t}}\bigg]
E_{f^t\omega}\bigg[\mathds{1}_{\big\{\frac{W_{t'}}{t'}\in A',\,t'\in\mathcal{T}\big\}}
\ee^{H_{f^t\omega,t'}}\bigg]=\mu_{\omega,t}(A)\mu_{f^t\omega,t'}(A').
\end{align}
In fact, if $T_n=t$ and $T_{n+n'}=t+t'$, then we have $W_t=\sum_{i=1}^n X_i$, $H_{\omega,t}=\sum_{i=1}^n v_{f^{T_{i-1}}\omega}(S_i)$, $W_{t+t'}-W_t=\sum_{i=1}^{n'} X_{n+i}$, and $H_{\omega,t+t'}-H_{\omega,t}=\sum_{i=1}^{n'} v_{f^{T_{n+i-1}}\omega}(S_{n+i})$. On the other hand, conditional on $T_n=t$, by (\ref{start_0}) the variables $S_{n+1},\ldots,S_{n+n'},T_n,\ldots,T_{n+n'},X_{n+1},\ldots,X_{n+n'}$ are independent of $S_1,\ldots,S_n,X_1,\ldots,X_n$ and jointly distributed as $S_1,\ldots,S_{n'},t+T_0,\ldots,t+T_{n'},X_1,\ldots,X_{n'}$ in the random environment $f^t\omega$.
\end{proof}

Given $t\in\N_0$ and a linear functional $\varphi\in\rew^\star$, consider now the random variable
\begin{equation*}
\omega\mapsto Z_{\omega,t}(\varphi) :=\int_\rew\ee^{t\varphi(w)}\mu_{\omega,t}(\dd w)
= E_\omega\Big[\mathds{1}_{\{t\in\mathcal{T}\}}\ee^{\varphi(W_t)+H_{\omega,t}}\Big].
\end{equation*}
Apart from normalisation, $Z_{\omega,t}(\varphi)$ is the moment generating function at $\varphi$ of the total reward $W_t$ with respect to the constrained pinning model $Q_{\omega,t}$. To develop our theory, we need both $Z_{\omega,t}(\varphi)$ and the following truncated version of $Z_{\omega,t}(\varphi)$, which has the advantage of involving finite expectations. With $r_{\omega,s}$ as in Assumption \ref{ass2}, put $Y_{\omega,i}:=X_i-r_{f^{T_{i-1}}\omega,S_i}$ for brevity and, for $A\in\Brew$, define the random variable
\begin{equation*}
\omega\mapsto \mathcal{E}_{\omega,t}(A,\varphi):=E_\omega\Big[\mathds{1}_{\{Y_{\omega,1}
\in A,\ldots,Y_{\omega,N_t}\in A,\,t\in\mathcal{T}\}}\ee^{\varphi(W_t)+H_{\omega,t}}\Big].
\end{equation*}
We have $\mathcal{E}_{\omega,t}(A,\varphi)\le Z_{\omega,t}(\varphi)$ and $\mathcal{E}_{\omega,t}(\rew,\varphi)=Z_{\omega,t}(\varphi)$. The identity in (\ref{start_0}) implies supermultiplicativity also for $\omega\mapsto\mathcal{E}_{\omega,t}(A,\varphi)$, as stated in the next lemma, whose proof is omitted because it is similar to that of Lemma \ref{super_mu}.

\begin{lemma}
\label{super_Z}
For $t,t'\in \N_0$, $A\in\Brew$, and $\varphi\in\rew^\star$, $\mathcal{E}_{\omega,t+t'}(A,\varphi) \ge \mathcal{E}_{\omega,t}(A,\varphi)\mathcal{E}_{f^t\omega,t'}(A,\varphi)$. In particular, $Z_{\omega,t+t'}(\varphi) \ge Z_{\omega,t}(\varphi)Z_{f^t\omega,t'}(\varphi)$.
\end{lemma} 


\paragraph{Fundamental limits.}

Let $t_{\mathrm{F}}$ be an integer larger than the Frobenius number associated with $\sigma_1,\ldots,\sigma_m$ in Assumption \ref{ass1}, the latter being finite since $\sigma_1,\ldots,\sigma_m$ are coprime integers. By the definition of Frobenius number, any $t\ge t_{\mathrm{F}}$ can be expressed as an integer conical combination of $\sigma_1,\ldots,\sigma_m$. With the points $x_s\in\rew$ of Assumption \ref{ass2}, let $\delta_{\mathrm{F}}$ be a number larger than $\max\{\|x_{\sigma_1}\|,\ldots,\|x_{\sigma_m}\|\}$. Together with the above supermultiplicativity properties, the following lemma sets the basis for the application of Kingman's subadditive ergodic theorem.

\begin{lemma}
\label{integrability_Z}
For $\varphi\in\rew^\star$ the following hold:\\
$(i)$ $\Ex[\min\{0,\log Z_{\bdot,t}(\varphi)\}]\ge\Ex[\min\{0,\log \mathcal{E}_{\bdot,t}(B_{0,\delta},\varphi)\}]>-\infty$ for all $t\ge t_{\mathrm{F}}$ and $\delta\ge\delta_\mathrm{F}$.\\
$(ii)$ $\sup_{t\in\N}\{\Ex[\max\{0,\frac{1}{t}\log Z_{\bdot,t}(0)\}]\}<+\infty$.\\
$(iii)$ $\sup_{t\in\N}\{\Ex[\max\{0,\frac{1}{t}\log \mathcal{E}_{\bdot,t}(B_{0,\delta},\varphi)\}]\}<+\infty$ for all $\delta>0$.\\
$(iv)$ for each $t\ge t_{\mathrm{F}}$ there exists a compact set $K\subset\rew$ such that $\Ex[\min\{0,\log \mathcal{E}_{\bdot,t}(K,\varphi)\}]>-\infty$.
\end{lemma}

By $(i)$ and $(iii)$, the random variables $\omega\mapsto\log Z_{\omega,t}(\varphi)$ and $\omega\mapsto\log\mathcal{E}_{\omega,t}(B_{0,\delta},\varphi)$ are integrable for all $t\ge t_{\mathrm{F}}$ and $\delta>0$. Moreover, by $(i)$ the expectation $\Ex[\log\mathcal{E}_{\bdot,t}(B_{0,\delta},\varphi)]$ is finite if $\delta\ge\delta_\mathrm{F}$. For $t\ge t_{\mathrm{F}}$, even the random variable $\omega\mapsto\log\mathcal{E}_{\omega,t}(K,\varphi)$ with $K\subset\rew$ compact is integrable. In fact, we have $K\subset B_{0,\delta}$ for some $\delta>0$ large enough, so that $\Ex[\max\{0,\log\mathcal{E}_{\bdot,t}(K,\varphi)\}]\le\Ex[\max\{0,\log\mathcal{E}_{\bdot,t}(B_{0,\delta},\varphi)\}]<+\infty$. Then, by $(iv)$, for each $t\ge t_{\mathrm{F}}$ there exists at least one compact set $K\subset\rew$ such that the expectation $\Ex[\log\mathcal{E}_{\bdot,t}(K,\varphi)]$ exists and is finite.

\begin{proof}[Proof of Lemma \ref{integrability_Z}]
By Assumptions \ref{ass2} and \ref{ass3}, there exists a real number $\eta\ge 0$ with the property that the expectations $\Ex[\sup_{s\in\N}\max\{0,\|r_{\bdot,s}\|-\eta s\}]$ and $\Ex[\sup_{s\in\N}\max\{0, v_{\bdot}(s)-\eta s\}]$ are finite.  For $t\in\N$, we have
\begin{align}
\nonumber
H_{\omega,t}:=\sum_{i=1}^{N_t}v_{f^{T_{i-1}}\omega}(S_i)&\le\sum_{i=1}^{N_t}
\max\big\{0,v_{f^{T_{i-1}}\omega}(S_i)-\eta S_i\big\}+\eta\sum_{i=1}^{N_t}S_i\\
\nonumber
&\le\sum_{i=1}^{N_t}\sup_{s\in\N}\max\big\{0,v_{f^{T_{i-1}}\omega}(s)-\eta s\big\}+\eta \,T_{N_t}\\
&\le\sum_{\tau=0}^{t-1}\sup_{s\in\N}\max\big\{0,v_{f^\tau\omega}(s)-\eta s\big\}+\eta t.
\label{H_bound}
\end{align}
Moreover, if $Y_{\omega,i}:=X_i-r_{f^{T_{i-1}}\omega,S_i}\in B_{0,\delta}$ for $i=1,\ldots,N_t$ with some $\delta>0$, then
\begin{align}
\nonumber
|\varphi(W_t)|
&=\Bigg|\sum_{i=1}^{N_t}\varphi(r_{f^{T_{i-1}}\omega,S_i})+\sum_{i=1}^{N_t}\varphi(Y_{\omega,i})\Bigg|
\le \|\varphi\|\sum_{i=1}^{N_t}\|r_{f^{T_{i-1}}\omega,S_i}\|+\|\varphi\|\sum_{i=1}^{N_t}\|Y_{\omega,i}\|\\
\nonumber
&\le \|\varphi\|\sum_{i=1}^{N_t}\sup_{s\in\N}\max\big\{0,\|r_{f^{T_{i-1}}\omega,s}\|-\eta s\big\}+\eta\, T_{N_t}+\|\varphi\|\delta N_t\\
\nonumber
&\le \|\varphi\|\sum_{\tau=0}^{t-1}\sup_{s\in\N}\max\big\{0,\|r_{f^\tau\omega,s}\|-\eta s\big\}+\eta t+\|\varphi\|\delta t.
\end{align}
In this way, for each Borel set $A\subseteq B_{0,\delta}$, we find the bounds
\begin{align}
\nonumber
\mathcal{E}_{\omega,t}(A,\varphi)
&:=E_\omega\Big[\mathds{1}_{\{Y_{\omega,1}\in A,\ldots,Y_{\omega,N_t}\in A,\,t\in\mathcal{T}\}}\ee^{\varphi(W_t)+H_{\omega,t}}\Big]\\
&\le\ee^{\|\varphi\|\sum_{\tau=0}^{t-1}\sup_{s\in\N}\max\{0,\|r_{f^\tau\omega,s}\|-\eta s\}+\eta t+\|\varphi\|\delta t}Z_{\omega,t}(0)
\label{integrability_Z_1}
\end{align}
and
\begin{equation}
\mathcal{E}_{\omega,t}(A,\varphi)\ge \ee^{-\|\varphi\|\sum_{\tau=0}^{t-1}
\sup_{s\in\N}\max\{0,\|r_{f^\tau\omega,s}\|-\eta s\}-\eta t-\|\varphi\|\delta t}
E_\omega\Big[\mathds{1}_{\{Y_{\omega,1}\in A,\ldots,Y_{\omega,N_t}\in A,\,t\in\mathcal{T}\}}\ee^{H_{\omega,t}}\Big].
\label{integrability_Z_2}
\end{equation}
The latter can be further manipulated by observing that, for positive integers $s_1,\ldots,s_n$ such that $s_1+\cdots+s_n=t$, we have
\begin{align}
\nonumber
E_\omega\Big[\mathds{1}_{\{Y_{\omega,1}\in A,\ldots,Y_{\omega,N_t}\in A,\,t\in\mathcal{T}\}}\ee^{H_{\omega,t}}\Big]
&\ge E_\omega\Big[\mathds{1}_{\{Y_{\omega,1}\in A,\,S_1=s_1,\ldots,Y_{\omega,n}\in A,\,S_n=s_n,\,N_t=n\}}\ee^{H_{\omega,t}}\Big]\\
&=\prod_{i=1}^n\ee^{v_{f^{t_{i-1}}\omega}(s_i)}p_{f^{t_{i-1}}\omega}(s_i)\,\lambda_{f^{t_{i-1}}\omega}
\big(r_{f^{t_{i-1}}\omega,s_i}+A\big|s_i\big),
\label{integrability_Z_3}
\end{align}
where $t_0:=0$ and $t_i:=s_1+\cdots+s_i$ for $i\ge 1$.

\medskip\noindent 
$(i)$ Fix an integer $t\ge t_\mathrm{F}$ and a real number $\delta\ge\delta_\mathrm{F}$. By the definition of $t_\mathrm{F}$, $t$ can be written as $t = \sum_{i=1}^ns_i$ with some $n\in\{1,\ldots,t\}$ and certain $s_1,\ldots,s_n \in \{\sigma_1,\ldots,\sigma_m\}$. By the definition of $\delta_\mathrm{F}$, there exist positive numbers $\delta_1,\ldots,\delta_n$ such that $B_{x_{s_i},\delta_i}\subset B_{0,\delta}$ for all $i$. Then, the bounds in (\ref{integrability_Z_2}) and (\ref{integrability_Z_3}) show that
\begin{align}
\nonumber
\Ex\Big[\min\big\{0,\log\mathcal{E}_{\bdot,t}(B_{0,\delta},\varphi)\big\}\Big]
&\ge -\|\varphi\|t\,\Ex\Big[\sup_{s\in\N}\max\big\{0,\|r_{\bdot,s}\|-\eta s\big\}\Big]-\eta t-\|\varphi\|\delta t\\
\nonumber
&\qquad + \sum_{i=1}^n\Ex\big[\min\{0,v_{\bdot}(s_i)\}\big]+\sum_{i=1}^n\Ex\big[\log p_{\bdot}(s_i)\big]\\
\nonumber
&\qquad +\sum_{i=1}^n\Ex\big[\log\lambda_{\bdot}(r_{\bdot,s_i}+B_{x_{s_i},\delta_i})\big].
\end{align}
This proves that $\Ex[\min\{0,\log Z_{\bdot,t}(\varphi)\}]\ge\Ex[\min\{0,\log \mathcal{E}_{\bdot,t}(B_{0,\delta},\varphi)\}]>-\infty$ thanks to Assumptions \ref{ass1}, \ref{ass2}, and \ref{ass3}.

\medskip\noindent 
$(ii)$ For $t\in\N$, the bound in (\ref{H_bound}) implies
\begin{equation*}
\log Z_{\omega,t}(0)=\log E_\omega[\mathds{1}_{\{t\in\mathcal{T}\}}\ee^{H_{\omega,t}}]\le\sum_{\tau=0}^{t-1}\sup_{s\in\N}\max\big\{0,v_{f^\tau\omega}(s)-\eta s\big\}+\eta t,
\end{equation*}
so that
\begin{equation*}
\Ex\bigg[\max\bigg\{0,\frac{1}{t}\log Z_{\bdot,t}(0)\bigg\}\bigg]\le \Ex\Big[\sup_{s\in\N}\max\big\{0,v_{\bdot}(s)-\eta s\big\}\Big]+\eta<+\infty.
\end{equation*}

\medskip\noindent 
$(iii)$ For $t\in\N$ and $\delta>0$, the bounds in (\ref{H_bound}) and (\ref{integrability_Z_1}) give
\begin{align}
\nonumber
\Ex\bigg[\max\bigg\{0,\frac{1}{t}\log \mathcal{E}_{\bdot,t}(B_{0,\delta},\varphi)\bigg\}\bigg]
&\le
\|\varphi\|\,\Ex\Big[\sup_{s\in\N}\max\{0,\|r_{\bdot,s}\|-\eta s\}\Big]+\eta +\|\varphi\|\delta\\
\nonumber
&\qquad +\Ex\Big[\sup_{s\in\N}\max\big\{0,v_{\bdot}(s)-\eta s\big\}\Big]+\eta<+\infty.
\end{align}

\medskip\noindent 
$(iv)$ Fix $t\ge t_\mathrm{F}$. As before, let $n\in\{1,\ldots,t\}$ and $s_1,\ldots,s_n \in \{\sigma_1,\ldots,\sigma_m\}$ be such that $t = \sum_{i=1}^ns_i$.  By Assumption \ref{ass2}, for each $i$ there exists a compact set $K_i\subset\rew$ that satisfies $\Ex[\log\lambda_{\bdot}(r_{\bdot,s_i}+K_i|s_i)]>-\infty$. Put $K:=\cup_{i=1}^nK_i$, and let $\delta>0$ be so large that $K\subset B_{0,\delta}$. Then, the bounds in (\ref{integrability_Z_2}) and (\ref{integrability_Z_3}), together with Assumptions \ref{ass1} and \ref{ass3}, show that
\begin{align}
\nonumber
\Ex\Big[\min\big\{0,\log\mathcal{E}_{\bdot,t}(K,\varphi)\big\}\Big]
&\ge -\|\varphi\|t\,\Ex\Big[\sup_{s\in\N}\max\big\{0,\|r_{\bdot,s}\|-\eta s\big\}\Big]-\eta t -\|\varphi\|\delta t\\
\nonumber
&\qquad + \sum_{i=1}^n\Ex\big[\min\{0,v_{\bdot}(s_i)\}\big]+\sum_{i=1}^n\Ex\big[\log p_{\bdot}(s_i)\big]\\
\nonumber
&\qquad +\sum_{i=1}^n\Ex\big[\log\lambda_{\bdot}(r_{\bdot,s_i}+K_i)\big]>-\infty.
\qedhere
\end{align}
\end{proof}

For $\varphi\in\rew^\star$, put
\begin{equation*}
z_o(\varphi):=\sup_{t\ge t_{\mathrm{F}}}\bigg\{\Ex\bigg[\frac{1}{t}\log Z_{\bdot,t}(\varphi)\bigg]\bigg\}.
\end{equation*}
H\"older's inequality shows that the functions that associate $\varphi$ with $\log Z_{\omega,t}(\varphi)$ are convex.  Hence the function $z_o$ that maps $\varphi$ in $z_o(\varphi)$ is also convex. Actually, it is proper convex since $(i)$ and $(ii)$ of Lemma \ref{integrability_Z} imply $z_o(\varphi)>-\infty$ for all $\varphi$ and $z_o(0)<+\infty$, respectively. In Section \ref{sec:Legendre} we will prove that $z_o$ is also lower semi-continuous, as it is the Legendre transform of a convex function. With Lemmas \ref{super_Z} and \ref{integrability_Z}, the following is immediate from $(\star)$ with $F_t(\omega)=-\log Z_{\omega,t}(\varphi)$.

\begin{corollary}
\label{Zc}
For $\varphi\in\rew^\star$, $\lim_{t\uparrow\infty}\frac{1}{t}\log Z_{\omega,t}(\varphi)=\lim_{t\uparrow\infty}\Ex[\frac{1}{t}\log Z_{\bdot,t}(\varphi)]=z_o(\varphi)$ $\pae$.
\end{corollary}

A similar result holds under truncation and will be used in Section \ref{sec:Legendre} to work out some technical steps. With Lemmas
\ref{super_Z} and \ref{integrability_Z}, the following follows from $(\star)$ with $F_t(\omega)=-\log \mathcal{E}_{\omega,t}(B_{0,\delta},\varphi)$.

\begin{corollary}
\label{Z_truncated}
For $\varphi\in\rew^\star$ and $\delta\ge\delta_\mathrm{F}$, the supremum $e:=\sup_{t\ge t_\mathrm{F}}\{\Ex[\frac{1}{t}\log \mathcal{E}_{\bdot,t}(B_{0,\delta},\varphi)]\}$ is finite and $\lim_{t\uparrow\infty}\frac{1}{t}\log\mathcal{E}_{\omega,t}(B_{0,\delta},\varphi)=e$ $\pae$.
\end{corollary}

We now go back to the random measures $\omega\mapsto\mu_{\omega,t}$. Although the sequence of term $\omega\mapsto\mu_{\omega,t}(C)$ with $C\in\Brew$ convex enjoys supermultiplicativity according to Lemma \ref{super_mu}, a result like Corollaries \ref{Zc} and \ref{Z_truncated} cannot be obtained from $(\star)$ because the condition $\Ex[\min\{0,\log\mu_{\bdot,t}(C)]>-\infty$ for all sufficiently large $t$ is not satisfied in general. Such a condition is met under the additional assumption that $\Ex[\log\lambda_{\bdot}(G|s)]>-\infty$ for all $s \in \N$ and $G \subseteq \rew$ is open and nonempty. The following lemma is a first step towards proving the quenched weak LDP for the family $\{\mu_{\omega,t}\}_{t\in\N}$ via an approximation argument. For $A\in\Brew$, put
\begin{equation*}
\mathcal{L}(A) := \sup_{t\ge t_{\mathrm{F}}} \bigg\{\Ex\bigg[\frac{1}{t}\log\mu_{\bdot,t}(A)\bigg]\bigg\},
\end{equation*}
where the expectation exists thanks to $(ii)$ of Lemma \ref{integrability_Z}, since $\mu_{\omega,t}(A)\le \mu_{\omega,t}(\rew)=Z_{\omega,t}(0)$. We note that $\mathcal{L}(A)\le \mathcal{L}(\rew)=z_o(0)$ for all $A\in\Brew$.

\begin{lemma}
\label{start}
Suppose that $\Ex[\log\lambda_{\bdot}(G|s)]>-\infty$ for all $s \in\N$ and $G \subseteq \rew$ is open and nonempty. Then the following hold for every $C,C'\subseteq\rew$ open and convex:\\
(i) $\lim_{t\uparrow\infty}\frac{1}{t}\log\mu_{\omega,t}(C)=\lim_{t\uparrow\infty}\Ex[\frac{1}{t}\log\mu_{\bdot,t}(C)]=\mathcal{L}(C)$ $\pae$.\\
(ii) $\mathcal{L}(\alpha C+\alpha'C')\ge\alpha \mathcal{L}(C)+\alpha' \mathcal{L}(C')$ for all rational numbers $\alpha,\alpha'> 0$ such that
$\alpha+\alpha'=1$.
\end{lemma}

\begin{proof}[Proof of Lemma \ref{start}]
To begin with, observe that if $X_i\in S_iC$ for $i=1,\ldots,N_t$ with some $t\in\N$, then $\frac{1}{t}\sum_{i=1}^{N_t} X_i\in C$ by convexity whenever $\sum_{i=1}^{N_t}S_i=T_{N_t}=t$, i.e., whenever $t\in\mathcal{T}$. It therefore follows that
\begin{align}
\nonumber
\mu_{\omega,t}(C)
&:=E_\omega\bigg[\mathds{1}_{\big\{\frac{W_t}{t}\in C,\,t\in\mathcal{T}\big\}} \ee^{H_{\omega,t}}\bigg]\\
\nonumber
&\ge E_\omega\Big[\mathds{1}_{\{X_1\in S_1C,\ldots,X_{N_t}\in S_{N_t}C,\,t\in\mathcal{T}\}} \ee^{H_{\omega,t}}\Big]\\
\nonumber
&=\sum_{n=1}^t\sum_{s_1\in\N}\cdots\sum_{s_n\in\N}\mathds{1}_{\{t_n=t\}}\prod_{i=1}^n
\ee^{v_{f^{t_{i-1}}\omega}(s_i)}p_{f^{t_{i-1}}\omega}(s_i)\,\lambda_{f^{t_{i-1}}\omega}(s_iC|s_i)\\
&\ge\prod_{\tau=0}^{t-1}\prod_{s=1}^t\lambda_{f^\tau\omega}(sC|s)\, Z_{\omega,t}(0),
\label{lower_bound_mu_Kingman}
\end{align}
where $t_0:=0$ and $t_i:=s_1+\cdots+s_i$ for $i\in\N$.

\medskip\noindent 
$(i)$ As this claim is trivially true when $C=\emptyset$, suppose that $C\ne\emptyset$. Under the additional assumption that $\Ex[\log\lambda_{\bdot}(G|s)]>-\infty$ for all $s \in \N$ and $G \subseteq \rew$ is open and nonempty, the bound in (\ref{lower_bound_mu_Kingman}) and $(i)$ of Lemma \ref{integrability_Z} ensure that $\Ex[\min\{0,\log\mu_{\bdot,t}(C)\}]>-\infty$ for all $t\ge t_{\mathrm{F}}$. In view of Lemma \ref{super_mu} and $(\star)$ with $F_t(\omega) = -\log\mu_{\omega,t}(C)$, we realise that $\lim_{t\uparrow\infty}\frac{1}{t}\log\mu_{\omega,t}(C)=\lim_{t\uparrow\infty}\Ex[\frac{1}{t}\log\mu_{\bdot,t}(C)]=\mathcal{L}(C)$ $\pae$.

\medskip\noindent 
$(ii)$ Given rational numbers $\alpha,\alpha'>0$ such that $\alpha+\alpha'=1$, there exists a $\beta\in\N$ with the property that $\alpha \beta$ and $\alpha' \beta$ are positive integers. For $n\in\N$, Lemma \ref{super_mu} with $t=\alpha\beta n$ and $t'=\alpha' \beta n$ shows that $\mu_{\omega,\beta n}(\alpha C+\alpha' C') \ge\mu_{\omega,\alpha \beta n}(C)\mu_{f^{\alpha \beta n}\omega,\alpha' \beta n}(C')$. Taking logarithms and expectations, dividing by $\beta n$, and letting $n \uparrow\infty$, we find $\mathcal{L}(\alpha C+\alpha' C')\ge\alpha \mathcal{L}(C)+\alpha' \mathcal{L}(C')$ by $(i)$ of the lemma.
\end{proof}


\subsection{A renewal equation in a random environment}
\label{sec:ren_eq}

In this section we prove that $z_o=z$, with $z$ the function defined in Section \ref{sec:mainresults} by the variational formula $z(\varphi) := \inf\{\zeta\in\Rl\colon\,\Upsilon_\varphi(\zeta)\le0\}$ with
\begin{equation*}
\Upsilon_\varphi(\zeta):= \infd_{R\in\mathcal{R}_+}\prob\,\mbox{-}\esssup_{\omega\in\Omega} 
\bigg\{\log E_\omega\Big[\ee^{\varphi(X_1)+v_\omega(S_1)-\zeta S_1+R(f^{S_1}\omega)-R(\omega)}
\mathds{1}_{\{S_1<\infty\}}\Big]\bigg\}. 
\end{equation*}
To this aim, we resort to the following renewal equation in a random environment: for $t \in \N$ and $\varphi \in \rew^\star$,
\begin{align}
\nonumber
Z_{\omega,t}(\varphi)
&=\sum_{s=1}^tE_\omega\Big[\mathds{1}_{\{S_1=s,\,t\in\mathcal{T}\}}\ee^{\varphi(W_t)+H_{\omega,t}}\Big]\\
\nonumber
&=\sum_{s=1}^tE_\omega\Big[\ee^{\varphi(X_1)+v_\omega(S_1)}\mathds{1}_{\{S_1=s\}}\Big]\,Z_{f^s\omega,t-s}(\varphi).
\end{align}
This equation is due to the fact that, conditional on $S_1=s$, the random variables $S_2,S_3,\ldots,X_2,X_3,\ldots$ are independent of $X_1$ and jointly distributed as $S_1,S_2,\ldots,X_1,X_2,\ldots$ in the environment $f^s\omega$. The equality $z_o=z$ demonstrates Proposition \ref{Zlim} thanks to Corollary \ref{Zc}, and the property that $z_o$ is proper convex, with $z_o(0)$ finite and lower semi-continuous. Lower semi-continuity of $z_o$ will be verified in Section \ref{sec:Legendre}.

Fix $\varphi\in\rew^\star$. We first verify that $z_o(\varphi)\le z(\varphi)$. To this aim we assume that $z(\varphi)<+\infty$, otherwise there is nothing to prove. Pick real numbers $\zeta>z(\varphi)$ and $\epsilon>0$. Since $\Upsilon_\varphi(\zeta) \le 0$ as $\zeta>z(\varphi)$, there exist a random variable $R\in\mathcal{R}_+$ and a set $\Omega_o\in\mathcal{F}$ such that $\prob[\Omega_o]=1$ and
\begin{equation*}
\log E_\omega\Big[\ee^{\varphi(X_1)+v_\omega(S_1)-\zeta S_1+R(f^{S_1}\omega)-R(\omega)}
\mathds{1}_{\{S_1<\infty\}}\Big]\le \epsilon
\end{equation*}
for every $\omega \in \Omega_o$. By changing $\Omega_o$ with $\cap_{t\in\N_0}f^{-t}\Omega_o$ if necessary, we may suppose that $\omega\in\Omega_o$ implies $f^t\omega\in\Omega_o$ for any $t\in\N_0$. Below we will prove that $Z_{\omega,t}(\varphi)\le \ee^{t(\zeta+\epsilon)+R(\omega)}$ for all $\omega\in\Omega_o$ and $t\in\N_0$. This gives $z_o(\varphi)\le\zeta +\epsilon$ since, by Corollary \ref{Zc}, there exists at least one point $\omega\in\Omega_o$ such that $\lim_{t\uparrow\infty}\frac{1}{t}\log Z_{\omega,t}(\varphi)=z_o(\varphi)$. The arbitrariness of $\zeta$ and $\epsilon$ demonstrates that $z_o(\varphi)\le z(\varphi)$.

We prove that $Z_{\omega,t}(\varphi)\le \ee^{t(\zeta+\epsilon)+R(\omega)}$ for all $\omega\in\Omega_o$ and $t\in\N_0$ by induction. The bound is true when $t=0$, as $Z_{\omega,0}(\varphi)=1$ and $R(\omega)>0$. Suppose that it holds for every $\omega\in\Omega_o$ up to $t-1$ with a positive $t$. Pick $\omega\in\Omega_o$. Then the renewal equation and the fact that $f^s\omega\in\Omega_o$ for $s=1,\ldots,t$ show that
\begin{align}
\nonumber
Z_{\omega,t}(\varphi)
&=\sum_{s=1}^tE_\omega\Big[\ee^{\varphi(X_1)+v_\omega(S_1)}\mathds{1}_{\{S_1=s\}}\Big]\,Z_{f^s\omega,t-s}(\varphi)\\
\nonumber
&\le\sum_{s=1}^tE_\omega\Big[\ee^{\varphi(X_1)+v_\omega(S_1)}\mathds{1}_{\{S_1=s\}}\Big]\,\ee^{(t-s)(\zeta+\epsilon)+R(f^s\omega)}\\
\nonumber
&=\ee^{t(\zeta+\epsilon)+R(\omega)}\sum_{s=1}^tE_\omega\Big[\ee^{\varphi(X_1)
+v_\omega(S_1)-\zeta S_1+R(f^{S_1}\omega)-R(\omega)}\mathds{1}_{\{S_1=s\}}\Big]\,
\ee^{-s\epsilon_i}\\
\nonumber
&\le \ee^{t(\zeta+\epsilon)+R(\omega)}E_\omega\Big[\ee^{\varphi(X_1)
+v_\omega(S_1)-\zeta S_1+R(f^{S_1}\omega)-R(\omega)}\mathds{1}_{\{S_1<\infty\}}\Big]\,
\ee^{-\epsilon}\le \ee^{t(\zeta+\epsilon)+R(\omega)}.
\end{align}

Next, let us show the opposite bound $z_o(\varphi) \ge z(\varphi)$. If $z_o(\varphi) = +\infty$, then there is nothing to prove. If instead $z_o(\varphi) < +\infty$, then pick a real number $\zeta > z_o(\varphi)$. By Corollary \ref{Zc}, the number $\zeta$ makes the series $\sum_{t \in \N_0} Z_{\omega,t}(\varphi)\,\ee^{-\zeta t}$ convergent $\pae$. We have $\sum_{t \in \N_0} Z_{\omega,t}(\varphi)\,\ee^{-\zeta t}>1$ $\pae$, since $Z_{\omega,0}(\varphi) = 1$ and $Z_{\omega,t}(\varphi)>0$ for all $t\ge t_{\mathrm{F}}$ $\pae$, the latter being implied by $(i)$ of Lemma \ref{integrability_Z}. Thus, the random variable $\omega \mapsto R(\omega) := \log\sum_{t \in \N_0}Z_{\omega,t}(\varphi)\,\ee^{-\zeta t}$ is almost surely finite and positive. Given any $\omega \in \Omega$ such that $0<R(\omega)< +\infty$, the renewal equation yields
\begin{align}
\nonumber
1 >\sum_{t \in \N}Z_{\omega,t}(\varphi)\,\ee^{-\zeta t-R(\omega)}
&=\sum_{s \in \N}\sum_{t\in\N_0}E_\omega\Big[\ee^{\varphi(X_1)+v_\omega(S_1)}
\mathds{1}_{\{S_1=s\}}\Big]\,Z_{f^s\omega,t}(\varphi)\,\ee^{-\zeta t-\zeta s-R(\omega)}\\
\nonumber
&=\sum_{s \in \N}E_\omega\Big[\ee^{\varphi(X_1)+v_\omega(S_1)}
\mathds{1}_{\{S_1=s\}}\Big]\,\ee^{-\zeta s+R(f^s\omega)-R(\omega)}\\
\nonumber
&= E_\omega\Big[\ee^{\varphi(X_1)+v_\omega(S_1)-\zeta S_1+R(f^{S_1}\omega)-R(\omega)}
\mathds{1}_{\{S_1<\infty\}}\Big].
\end{align}
This shows that $\Upsilon_\varphi(\zeta)<0$, and consequently $z(\varphi)\le\zeta$. We get $z(\varphi)\le z_o(\varphi)$ by letting $\zeta\downarrow z_o(\varphi)$.


\subsection{The weak LDP} 
\label{sec:WLDPc} 

The quenched weak LDP for the family $\{\mu_{\omega,t}\}_{t\in\N}$ is obtained after replacing the probability measure $\lambda_\omega(\bdot|s)$ by convenient probability measures $\lambda_\omega^1(\bdot|s),\lambda_\omega^2(\bdot|s),\ldots$ that satisfy the integrability condition of Lemma \ref{start}. The following lemma introduces such probability measures, resulting from a coupling argument that underlies an exponential
approximation.

\begin{lemma}
\label{coupling}
There exist a real number $L\ge 0$ and, for $s\in\N$, a sequence $\Lambda_{\omega}^1(\bdot|s),\Lambda_{\omega}^2(\bdot|s),\ldots$ of probability measures on $\mathcal{B}(\rew\times\rew)$ such that the following hold for all $k$:\\ 
$(i)$ $\lambda_\omega(A\times\rew|s)=\Lambda_{\omega}^k(A\times\rew|s)$ and $\omega\mapsto\lambda_\omega^k(A|s):=\Lambda_{\omega}^k(A\times\rew |s)$ is measurable for each $A\in\Brew$.\\ 
$(ii)$ $\int_{\rew\times\rew} \ee^{k\|x-x'\|}\Lambda_{\omega}^k(\dd(x,x')|s)\le \ee^L$.\\ 
$(iii)$ $\Ex[\log\lambda_{\bdot}^k(G|s)] > -\infty$ for all $G \subseteq \rew$ open and nonempty.
\end{lemma}

\begin{proof}[Proof of Lemma \ref{coupling}]
Let $\mathcal{S} := \{u_n\}_{n\in \N}$ be a countable dense subset of $\rew$, which exists by separability, and let
\begin{equation*}
q_n :=\frac{\ee^{-n-\|u_n\|}}{\sum_{i\in\N}\ee^{-i-\|u_i\|}}
\end{equation*}
define a probability mass function on $\N$. We have $\sum_{n\in\N}\ee^{\|u_n\|}q_n<+\infty$. Let $\{e_1,\ldots,e_d\}$ be a basis of the subspace $\mathcal{V}$ in Assumption \ref{ass2}, and let $g$ be the function that maps $\zeta:=(\zeta_1,\ldots,\zeta_d)\in\Rl^d$ in $g(\zeta):=\sum_{i=1}^d \zeta_ie_i\in\mathcal{V}$. With $g$, construct on $\mathcal{B}(\mathcal{V})$ the probability measure
\begin{equation*}
\rho:=\frac{\int_{\Rl^d}\mathds{1}_{\{g(\zeta)\in \,\bdot\,\}}\ee^{-2\|g(\zeta)\|}\dd \zeta}{\int_{\Rl^d}\ee^{-2\|g(\zeta)\|}\dd \zeta},
\end{equation*}
where $\dd \zeta$ is the Lebesgue measure. We have $\int_{\mathcal{V}}e^{\|v\|}\rho(\dd v)<+\infty$ and $\rho(v+A)\ge \ee^{-2\|v\|}\rho(A)$ for all $v\in\mathcal{V}$ and $A\in\mathcal{B}(\mathcal{V})$, the latter being a consequence of the translation invariance of the Lebesgue measure. We claim that the lemma holds with the number
\begin{equation*}
L:=\log\sum_{n\in\N}\ee^{\|u_n\|}q_n+\log\int_{\mathcal{V}}\ee^{\|v\|}\rho(\dd v)
\end{equation*}
and the probability measures $\Lambda_{\omega}^k(\bdot|s)$ defined for $k\in\N$ and $\mathsf{A}\in\mathcal{B}(\rew\times\rew)$ by
\begin{equation*}
\Lambda_{\omega}^k(\mathsf{A}|s):=\sum_{n\in\N}q_n\int_{\mathcal{V}}
\bigg[\int_\rew\mathds{1}_{\big\{(x,u_n/k+x-v/k)\in \mathsf{A}\big\}} \lambda_\omega(\dd x|s)\bigg]\rho(\dd v).
\end{equation*}
For $A\in\Brew$, we have
\begin{equation*}
\lambda_\omega^k(A|s):=\Lambda_{\omega}^k(\rew\times A|s)
=\sum_{n\in\N}q_n\int_{\mathcal{V}}\lambda_\omega\big(v/k+A-u_n/k\big|s\big)\rho(\dd v).
\end{equation*}

\medskip\noindent 
$(i)$ and $(ii)$ Fix $s$ and $k$ in $\N$. It is manifest that $\Lambda_{\omega}^k(A\times\rew|s)=\lambda_\omega(A|s)$ for all $A\in\Brew$, and that
\begin{equation*}
\int_{\rew\times\rew}\ee^{k\|x-x'\|}\Lambda_{\omega}^k(\dd (x,x')|s)
=\sum_{n\in\N}q_n\int_{\mathcal{V}}\ee^{\|u_n-v\|}\rho(\dd v)
\le \sum_{n\in\N}\ee^{\|u_n\|}q_n\int_{\mathcal{V}}\ee^{\|v\|}\rho(\dd v)=\ee^L.
\end{equation*}

\medskip\noindent 
$(iii)$ Fix $s$ and $k$ in $\N$ and pick a set $G\subseteq\rew$ open and nonempty, $x\in G$, and $\delta>0$ such that $B_{x,3\delta} \subseteq G$. Recalling that $\mathcal{S}$ is dense in $\rew$, let $m\in\N$ be such that $\|x-u_m/k\|<\delta$. Note that if $y\in B_{r_{\omega,s}+x_s,\delta}$ and $v\in B_{kr_{\omega,s}+kx_s,k\delta}$ with $r_{\omega,s}$ and $x_s$ as in Assumption \ref{ass2}, then
\begin{equation*}
\big\|y-v/k-x+u_m/k\big\| \le\big\|y-r_{\omega,s}-x_s\big\|+\big\|r_{\omega,s}+x_s-v/k\big\|+\big\|x-u_m/k\big\| < 3\delta,
\end{equation*}
so that $r_{\omega,s}+B_{x_s,\delta}=B_{r_{\omega,s}+x_s,\delta} \subseteq B_{v/k+x-u_m/k,3\delta}$. It follows that
\begin{align}
\nonumber
\lambda_{\omega}^k(G|s)\ge \lambda_{\omega}^k(B_{x,3\delta}|s)
&\ge q_m\int_{\mathcal{V}\cap B_{kr_{\omega,s}+kx_s,k\delta}}
\lambda_\omega\big(B_{v/k+x-u_m/k,3\delta}\big|s\big)\rho(\dd v)\\
\nonumber
&\ge q_m\,\lambda_\omega(r_{\omega,s}+B_{x_s,\delta}|s)\,\rho(\mathcal{V}\cap B_{kr_{\omega,s}+kx_s,k\delta}).
\end{align}
In this way, since $\rho(\mathcal{V}\cap B_{kr_{\omega,s}+kx_s,k\delta})=\rho(kr_{\omega,s}+\mathcal{V}\cap B_{kx_s,k\delta})\ge \ee^{-2k\|r_{\omega,s}\|}\rho(\mathcal{V}\cap B_{kx_s,k\delta})$ because $r_{\omega,s}\in\mathcal{V}$, we find
\begin{align}
\nonumber
\Ex\big[\log\lambda_{\bdot}^k(G|s)\big]&\ge \log q_m+\Ex\big[\log\lambda_{\bdot}(r_{\bdot,s}+B_{x_s,\delta}|s)\big]\\
\nonumber
&\qquad -2k\,\Ex\Big[\sup_{t\in\N}\max\big\{0,\|r_{\bdot,t}\|-\eta t\big\}\Big]-2k\eta+\log\rho(\mathcal{V}\cap B_{kx_s,k\delta})>-\infty
\end{align}
by Assumption \ref{ass2}.
\end{proof}

Thanks to $(iii)$ of Lemma \ref{coupling}, Lemma \ref{start} applies to the random measures $\omega\mapsto\mu_{\omega,t}^k$ associated with
$\lambda_{\omega}^k(\bdot|s)$: for each open convex set $C$
\begin{equation*}
\lim_{t\uparrow\infty}\frac{1}{t}\log\mu_{\omega,t}^k(C)=\mathcal{L}^k(C):=\sup_{t\ge
t_{\mathrm{F}}}\bigg\{\Ex\bigg[\frac{1}{t}\log\mu_{\bdot,t}^k(C)\bigg]\bigg\} \qquad \pae.
\end{equation*}
The function $J_o$ that maps $w\in\rew$ in
\begin{equation*}
J_o(w) := -\inf_{\delta>0}\bigg\{\liminf_{k\uparrow\infty}\mathcal{L}^k(B_{w,\delta})\bigg\}
\end{equation*}
turns out to be the rate function in a quenched weak LDP for the family $\{\mu_{\omega,t}\}_{t\in\N}$, as we will show below. The following lemma provides two general properties of $J_o$. We note that $J_o(w) \ge -z_o(0)>-\infty$.

\begin{lemma}
\label{prop_J}
The function $J_o$ is lower semi-continuous and convex.
\end{lemma}
  
\begin{proof}[Proof of Lemma \ref{prop_J}]
Pick $w\in\rew$ and a sequence $\{w_i\}_{i \in \N}$ in $\rew$ that converges to $w$. Given $\delta>0$, the monotonicity of $\mathcal{L}^k$ inherited from the measures $\mu_{\omega,t}^k$ entails that $-J_o(w_i) \le \liminf_{k\uparrow\infty} \mathcal{L}^k(B_{w_i,\delta/2}) \le \liminf_{k\uparrow\infty}\mathcal{L}^k(B_{w,\delta})$ for all sufficiently large $i$ satisfying $B_{w_i,\delta/2}\subseteq B_{w,\delta}$. This implies the bound $\liminf_{i\uparrow\infty}J_o(w_i) \ge -\liminf_{k\uparrow\infty}\mathcal{L}^k(B_{w,\delta})$ and proves the lower semi-continuity of $J_o$ by the arbitrariness of $\delta$.

As far as the convexity of $J_o$ is concerned, due to lower semi-continuity it suffices to verify that $J_o(\alpha w+\alpha' w')\le \alpha J_o(w)+\alpha' J_o(w')$ for every fixed $w,w' \in \rew$ and rational numbers $\alpha,\alpha'> 0$ such that $\alpha+\alpha'=1$. Given $\delta>0$ and observing that $B_{\alpha w+\alpha' w',\delta}\supseteq \alpha B_{w,\delta}+\alpha' B_{w',\delta}$, $(ii)$ of Lemma \ref{start} shows that $\mathcal{L}^k(B_{\alpha w+\alpha' w',\delta})\ge \mathcal{L}^k(\alpha B_{w,\delta}+\alpha' B_{w',\delta})\ge \alpha\mathcal{L}^k(B_{w,\delta})+\alpha'\mathcal{L} ^k(B_{w',\delta})$ for all $k\in\N$. From here, letting $k\uparrow\infty$, we get $-J_o(\alpha w+\alpha'w')\ge -\alpha J_o(w)-\alpha'J_o(w')$ because $\delta$ is arbitrary.
\end{proof}

We next show that $\pae$ the family $\{\mu_{\omega,t}\}_{t\in\N}$ satisfies a weak LDP with rate function $J_o$ by resorting to a coupling with new rewards distributed according to $\lambda_\omega^k(\bdot|s)$. In fact, with the sequence of probability measures $\Lambda_{\omega}^1(\bdot|s),\Lambda_{\omega}^2(\bdot|s),\ldots$ introduced by Lemma \ref{coupling}, we consider waiting times $S_1,S_2,\ldots$ and pairs of rewards $(X_1,X_1'),(X_2,X_2'),\ldots$ distributed according to the joint law
\begin{equation*}
\mathsf{P}_\omega^k\Big[S_1=s_1,\ldots,S_n=s_n,\,(X_1,X_1')\in \mathsf{A}_1,\ldots,(X_n,X_n')\in \mathsf{A}_n\Big]
= \prod_{i=1}^n p_{f^{t_{i-1}}\omega}(s_i)\,\Lambda_{f^{t_{i-1}}\omega}^k(\mathsf{A}_i|s_i)
\end{equation*}
for $n\in \N$, $s_1,\ldots,s_n \in \overline{\N}$, and $\mathsf{A}_1,\ldots,\mathsf{A}_n\in \mathcal{B}(\rew\times\rew)$, where $t_0 := 0$ and $t_i := s_1+\cdots+s_i$ for $i\in \N$. Let $W_t':=\sum_{i=1}^{N_t}X_i'$ be the total reward associated with the new rewards $X_1',X_2',\ldots$ and denote by $\mathsf{E}_\omega^k$ expectation under $\mathsf{P}_\omega^k$. Since $\lambda_{\omega}(\bdot|s)$ and $\lambda_{\omega}^k(\bdot|s)$ are
marginals of $\Lambda_{\omega}^k(\bdot|s)$, we have
\begin{equation*}
\mu_{\omega,t}=\mathsf{E}_\omega^k\bigg[\mathds{1}_{\big\{\frac{W_t}{t}\in \,\bdot\,,\,t\in\mathcal{T}\big\}} \ee^{H_{\omega,t}}\bigg],
\qquad
\mu_{\omega,t}^k=\mathsf{E}_\omega^k\bigg[\mathds{1}_{\big\{\frac{W_t'}{t}\in \,\bdot\,,\,t\in\mathcal{T}\big\}} \ee^{H_{\omega,t}}\bigg].
\end{equation*}
Importantly, the properties of $\Lambda_{\omega}^k(\bdot|s)$ make $W_t'$ an exponential approximation of $W_t$, in the sense that, for
$t\in\N_0$, $k\in\N$, and $\delta>0$,
\begin{equation}
\mathsf{E}_\omega^k\Big[\mathds{1}_{\{\|W_t-W_t'\|>\delta t,\,t\in\mathcal{T}\}} \ee^{H_{\omega,t}}\Big]
\le\ee^{Lt-k\delta t}Z_{\omega,t}(0)
\label{Chernoff_k}
\end{equation}
with $L$ as in Lemma \ref{coupling}. Indeed, a Chernoff-type bound, in combination with $(ii)$ of Lemma \ref{coupling}, gives
\begin{align}
\nonumber
\mathsf{E}_\omega^k&\Big[\mathds{1}_{\{\|W_t-W_t'\|>\delta t,\,t\in\mathcal{T}\}} \ee^{H_{\omega,t}}\Big]\\
\nonumber
&\le\ee^{-k\delta t}\mathsf{E}_\omega^k\Big[\mathds{1}_{\{t\in\mathcal{T}\}} \ee^{k\|W_t-W_t'\|+H_{\omega,t}}\Big]\\
\nonumber
&\le \ee^{-k\delta t}\mathsf{E}_\omega^k\Big[\mathds{1}_{\{t\in\mathcal{T}\}} \ee^{\sum_{i=1}^{N_t} k\|X_i-X_i'\|+H_{\omega,t}}\Big]\\
\nonumber
&=\ee^{-k\delta t}\sum_{n=1}^t\sum_{s_1\in\N}\cdots\sum_{s_n\in\N}\mathds{1}_{\{t_n=t\}}\prod_{i=1}^n
\ee^{v_{f^{t_{i-1}}\omega}(s_i)}p_{f^{t_{i-1}}\omega}(s_i)\int_{\rew\times\rew}\ee^{k\|x-x'\|}\Lambda_{f^{t_{i-1}}\omega}^k(\dd (x,x')|s_i)\\
\nonumber
&\le\ee^{-k\delta t}\sum_{n=1}^t\sum_{s_1\in\N}\cdots\sum_{s_n\in\N}\mathds{1}_{\{t_n=t\}}\ee^{Ln}\prod_{i=1}^n
\ee^{v_{f^{t_{i-1}}\omega}(s_i)}p_{f^{t_{i-1}}\omega}(s_i)\le \ee^{Lt-k\delta t}Z_{\omega,t}(0).
\end{align}
The bound in (\ref{Chernoff_k}) is the basis of the following proposition, which states the desired LDP.

\begin{proposition}
\label{prop:ldp_Jo}
The following bounds hold $\pae$:\\
$(i)$ $\liminf_{t\uparrow\infty}\frac{1}{t}\log\mu_{\omega,t}(B_{w,\delta})\ge-J_o(w)$ for all $w\in\rew$ and $\delta>0$.\\
$(ii)$ $\inf_{\delta>0}\{\limsup_{t\uparrow\infty}\frac{1}{t}\log\mu_{\omega,t}(B_{w,\delta})\}\le-J_o(w)$ for all $w\in\rew$.\\
Consequently, $\pae$ the family $\{\mu_{\omega,t}\}_{t\in\N}$ satisfies the weak LDP with rate function $J_o$.
\end{proposition}

\begin{proof}[Proof of Proposition \ref{prop:ldp_Jo}]
Suppose for the moment that $(i)$ and $(ii)$ hold for a certain $\omega\in\Omega$. Then, the family $\{\mu_{\omega,t}\}_{t\in\N}$ is shown to satisfy the weak LDP with rate function $J_o$ as follows. Given an open set $G\subseteq\rew$, a point $w\in G$, and a real number $\delta>0$ such that $B_{w,\delta}\subseteq G$, $(i)$ implies $\liminf_{t\uparrow\infty}\frac{1}{t}\log\mu_{\omega,t}(G)\ge\liminf_{t\uparrow\infty}\frac{1}{t}\log\mu_{\omega,t}(B_{w,\delta})\ge-J_o(w)$. The arbitrariness of $w$ yields the large deviation lower bound $\liminf_{t\uparrow\infty}\frac{1}{t}\log\mu_{\omega,t}(G)\ge-\inf_{w\in G}J_o(w)$. At the same time, given a compact set $K\subset\rew$ and a real number $\lambda<\inf_{w\in K}J_o(w)$, by $(ii)$ for each $w\in K$ there exists $\delta_w>0$ with the property $\limsup_{t\uparrow\infty} \frac{1}{t}\log\mu_{\omega,t}(B_{w,\delta_w})\le-\lambda$. As $\{B_{w,\delta_w}\}_{w\in K}$ is an open cover of $K$, there exist finitely many points $w_1,\ldots,w_n$ in $K$ such that $K \subseteq \cup_{i=1}^n B_{w_i,\delta_{w_i}}$. Thus, $\mu_{\omega,t}(K)\le\sum_{i=1}^n\mu_{\omega,t}(B_{w_i,\delta_{w_i}})$ for all $t\in\N$, giving $\limsup_{t\uparrow\infty}\frac{1}{t}\log\mu_{\omega,t}(K)\le-\lambda$. The arbitrariness of $\lambda$ implies the large deviation upper bound $\limsup_{t\uparrow\infty} \frac{1}{t}\log\mu_{\omega,t}(K) \le-\inf_{w\in K} J_o(w)$.

Let us next verify $(i)$ and $(ii)$. By separability, $\rew$ contains a countable dense subset $\mathcal{S}$. Denote by $\mathbb{Q}_+$ the set of positive rational numbers. Thanks to $(iii)$ of Lemma \ref{coupling}, Lemma \ref{start} ensures that there exists $\Omega_o\in\mathcal{F}$ with $\prob[\Omega_o]=1$ such that $\lim_{t\uparrow\infty} \frac{1}{t} \log\mu_{\omega,t}^k(B_{u,\alpha}) = \mathcal{L}^k(B_{u,\alpha})$ for all $k\in\N$, $u\in\mathcal{S}$, $\alpha\in\mathbb{Q}_+$, and $\omega\in\Omega_o$. By Corollary \ref{Zc}, we may also suppose that $\lim_{t\uparrow\infty} \frac{1}{t} \log Z_{\omega,t}(0) = z_o(0) < +\infty$ for every $\omega\in\Omega_o$. We prove that $(i)$ and $(ii)$ hold for any given $\omega\in\Omega_o$. To this aim, we make use of the bound in (\ref{Chernoff_k}) to state that, for $w\in\rew$ and $\delta>0$,
\begin{align}
\nonumber
\mu_{\omega,t}(B_{w,2\delta})
&=\mathsf{E}_\omega^k\bigg[\mathds{1}_{\big\{\frac{W_t}{t}\in B_{w,2\delta},\,t\in\mathcal{T}\big\}} \ee^{H_{\omega,t}}\bigg]\\
\nonumber
&\ge\mathsf{E}_\omega^k\bigg[\mathds{1}_{\big\{\frac{W_t'}{t}\in B_{w,\delta},\,\|W_t-W_t'\|\le\delta t,\,t\in\mathcal{T}\big\}} 
\ee^{H_{\omega,t}}\bigg]\\
\nonumber
&\ge \mathsf{E}_\omega^k\bigg[\mathds{1}_{\big\{\frac{W_t'}{t}\in B_{w,\delta},\,t\in\mathcal{T}\big\}} \ee^{H_{\omega,t}}\bigg]
-\mathsf{E}_\omega^k\bigg[\mathds{1}_{\big\{\|W_t-W_t'\|>\delta t,\,t\in\mathcal{T}\big\}} \ee^{H_{\omega,t}}\bigg]\\
&\ge \mu_{\omega,t}^k(B_{w,\delta})-\ee^{Lt-k\delta t}Z_{\omega,t}(0).
\label{lower_muk}
\end{align}
Exchanging $W_t$ and $W_t'$, we also have
\begin{equation}
\mu_{\omega,t}(B_{w,\delta})\le\mu_{\omega,t}^k(B_{w,2\delta})+\ee^{Lt-k\delta t}Z_{\omega,t}(0).
\label{upper_muk}
\end{equation}

\medskip\noindent 
$(i)$ Fix $w\in\rew$ such that $J_o(w)<+\infty$ (otherwise there is nothing to prove) and $\delta>0$. Pick $\epsilon>0$. We show that
\begin{equation}
\label{ldp_jo_lower}
\liminf_{t\uparrow\infty} \frac{1}{t}\log\mu_{\omega,t}(B_{w,\delta})\ge-J_o(w)-2\epsilon,
\end{equation}
which yields $(i)$ by the arbitrariness of $\epsilon$. Let $\alpha\in\mathbb{Q}_+$ and $u\in\mathcal{S}$ have the properties $5\alpha\le\delta$ and $\|u-w\|<\alpha$, so that $B_{w,\alpha}\subseteq B_{u,2\alpha}\subseteq B_{u,4\alpha}\subseteq B_{w,5\alpha}\subseteq B_{w,\delta}$. By construction, $J_o(w)$ satisfies $\liminf_{k\uparrow\infty}\mathcal{L}^k(B_{w,\alpha})\ge-J_o(w)$, which allows us to find $k\in\N$ so large that $\mathcal{L}^k(B_{w,\alpha})\ge -J_o(w)-\epsilon$ and $z_o(0)\le 2k\alpha-L-J_o(w)-4\epsilon$ with $L$ as in Lemma \ref{coupling}. Thus, since $\lim_{t\uparrow\infty}\frac{1}{t}\ln\mu_{\omega,t}^k(B_{u,2\alpha})=\mathcal{L}^k(B_{u,2\alpha})\ge\mathcal{L}^k(B_{w,\alpha})$ and $\lim_{t\uparrow\infty}\frac{1}{t}\ln Z_{\omega,t}(0)=z_o(0)$, we can state that $\frac{1}{t}\ln\mu_{\omega,t}^k(B_{u,2\alpha})\ge -J_o(w)-2\epsilon$ and $\frac{1}{t}\ln Z_{\omega,t}(0)\le 2k\alpha-L-J_o(w)-3\epsilon$ for all sufficiently large $t$. For such $t$, the bound in (\ref{lower_muk}) gives
\begin{equation*}
\mu_{\omega,t}(B_{w,\delta})\ge\mu_{\omega,t}(B_{u,4\alpha})\ge\mu_{\omega,t}^k(B_{u,2\alpha})
-\ee^{Lt-2k\alpha t}Z_{\omega,t}(0)\ge e^{-J_o(w)t-2\epsilon t}(1-\ee^{-\epsilon t}),
\end{equation*}
which demonstrates (\ref{ldp_jo_lower}).

\medskip\noindent 
$(ii)$ Pick $w\in\rew$ and a real number $\lambda<J_o(w)$. We prove that
\begin{equation}
\label{ldp_jo_upper}
\inf_{\delta>0}\bigg\{\limsup_{t\uparrow\infty} \frac{1}{t}\log\mu_{\omega,t}(B_{w,\delta})\bigg\}\le-\lambda,
\end{equation}
which demonstrates $(ii)$ thanks to the arbitrariness of $\lambda$. As $J_o(w):= -\inf_{\delta>0}\{\liminf_{k\uparrow\infty}\mathcal{L}^k(B_{w,\delta})\}$, there exists an $\alpha\in\mathbb{Q}_+$ such that $\liminf_{k\uparrow\infty}\mathcal{L}^k(B_{w,4\alpha})<-\lambda$. Hence there exists a $k\in\N$ so large that $\mathcal{L}^k(B_{w,4\alpha}) \le -\lambda$ and $z_o(0)\le k\alpha-L-\lambda$. Finally, there exists a $u\in\mathcal{S}$ with the property $\|w-u\|<\alpha$. We have $B_{w,2\alpha}\subseteq B_{u,3\alpha}\subseteq B_{w,4\alpha}$. The bound in (\ref{upper_muk}) gives
$\mu_{\omega,t}(B_{w,\alpha})\le\mu_{\omega,t}^k(B_{w,2\alpha})+\ee^{Lt-k\alpha t}Z_{\omega,t}(0)\le\mu_{\omega,t}^k(B_{u,3\alpha})+\ee^{Lt-k\alpha t}Z_{\omega,t}(0)$ for all $t\in\N$, so that
\begin{align}
\nonumber
\limsup_{t\uparrow\infty} \frac{1}{t}\log\mu_{\omega,t}(B_{w,\alpha})&\le\max\Big\{\mathcal{L}^k(B_{u,3\alpha}),\,L-k\alpha+z_o(0)\Big\}\\
\nonumber
&\le\max\Big\{\mathcal{L}^k(B_{w,4\alpha}),\,L-k\alpha+z_o(0)\Big\}\le-\lambda.
\end{align}
This proves (\ref{ldp_jo_upper}).
\end{proof}


\subsection{Legendre transform of the free energy}
\label{sec:Legendre}

In this section we show that $J_o$ is the Legendre transform $z_o^\star$ of $z_o$ defined, for $w\in\rew$, by $z_o^\star(w):=\sup_{\varphi\in\rew^\star}\{\varphi(w)-z_o(\varphi)\}$. Since $z_o=z$, this proves that $J_o$ is the rate function $J$ introduced in Section \ref{sec:mainresults} as the Legendre transform of $z$. Thus, Proposition \ref{prop:ldp_Jo} demonstrates Theorem \ref{WLDPc}.

In order to show that $J_o=z_o^\star$, it suffices to verify that $J_o^\star=z_o$ with $J_o^\star$ the Legendre transform of $J_o$ defined, for $\varphi\in\rew^\star$, by $J_o^\star(\varphi):=\sup_{w\in\rew}\{\varphi(w)-J_o(w)\}$. In fact, a general result from convex analysis states that $J_o(w)=J_o^{\star\star}(w):=\sup_{\varphi\in\rew^\star}\{\varphi(w)-J_o^\star(\varphi)\}$ if $J_o$ is proper convex and lower semi-continuous (see \cite{zalinescubook}, Theorem 2.3.3). We already know from Lemma \ref{prop_J} that $J_o$ is convex and lower semi-continuous. If $J_o^\star=z_o$, then $J_o$ is proper convex because $\inf_{w\in\rew}J_o(w)=-J_o^\star(0)=-z_o(0)$ with $z_o(0)$ finite implies that $J_o(w)>-\infty$ for all $w\in\rew$ and $J_o(w)<+\infty$ for some $w\in\rew$. We note that the identity $J_o^\star=z_o$ also proves that $z_o$ is lower semi-continuous, as claimed in Section \ref{sec:super}.

To verify that $J_o^\star(\varphi)=z_o(\varphi)$ for $\varphi\in\rew^\star$, we first demonstrate that $J_o^\star(\varphi)\le z_o(\varphi)$, and afterwards that $J_o^\star(\varphi)\ge z_o(\varphi)$. The latter is harder to obtain than the former.


\paragraph{Lower bound.}

Fix $\varphi\in\rew^\star$, $w\in\rew$, and $\delta>0$. Note that if $\frac{W_t}{t}\in B_{w,\delta}$, then $\varphi(W_t)-t\varphi(w)=\varphi(W_t-tw) \ge -\|\varphi\|\|W_t-tw\| \ge -\|\varphi\|\delta t$, namely $\varphi(W_t)\ge t\varphi(w)-t\|\varphi\|\delta$. It follows that
\begin{align}
\nonumber
E_\omega\Big[\mathds{1}_{\{t\in\mathcal{T}\}} \ee^{\varphi(W_t)+H_{\omega,t}}\Big]
&\ge E_\omega\bigg[\mathds{1}_{\big\{\frac{W_t}{t}\in B_{w,\delta},\,t\in\mathcal{T}\big\}}\ee^{\varphi(W_t)+H_{\omega,t}}\bigg]\\
\nonumber
&\ge \ee^{t\varphi(w)-t\|\varphi\|\delta}\,E_\omega\bigg[\mathds{1}_{\big\{\frac{W_t}{t}\in B_{w,\delta},\,t\in\mathcal{T}\big\}}
\ee^{H_{\omega,t}}\bigg]
= \ee^{t\varphi(w)-t\|\varphi\|\delta}\,\mu_{\omega,t}(B_{w,\delta})
\end{align}
for all $t\in\N$. Taking logarithms, dividing by $t$, and letting $t\uparrow\infty$, we get $z_o(\varphi) \ge \varphi(w)-J_o(w)-\|\varphi\|\delta$ thanks to Corollary \ref{Zc} and $(i)$ of Proposition \ref{prop:ldp_Jo}. Thus, appealing to the arbitrariness of $w$ and $\delta$, we find the bound $z_o(\varphi) \ge \sup_{w\in\rew}\{\varphi(w)-J_o(w)\} =: J_o^\star(\varphi)$.


\paragraph{Upper bound.}

Pick $\varphi\in\rew^\star$ and real numbers $\zeta\le z_o(\varphi)$ and $\epsilon>0$. The goal is to show that $\zeta\le J_o^\star(\varphi)+3\epsilon$. This gives $J_o^\star(\varphi)\ge z_o(\varphi)$ by the arbitrariness of $\zeta$ and $\epsilon$. The random variables $\omega\mapsto \mathcal{E}_{\omega,t}(A,\varphi)$ defined in Section \ref{sec:super} come into play here. 

As $\zeta\le z_o(\varphi)$, Corollary \ref{Zc} implies that there exists an integer $\tau\ge t_{\mathrm{F}}$ such that
\begin{equation*}
\zeta\le \Ex\bigg[\frac{1}{\tau}\log Z_{\bdot,\tau}(\varphi)\bigg]+\epsilon.
\end{equation*}
The monotone converge theorem shows that $\lim_{\delta\uparrow\infty}\mathcal{E}_{\omega,\tau}(B_{0,\delta},\varphi)=Z_{\omega,\tau}(\varphi)$. Since $\Ex[\log \mathcal{E}_{\bdot,\tau}(B_{0,\delta},\varphi)]$ exists and is finite for all $\delta\ge\delta_\mathrm{F}$ by Lemma \ref{integrability_Z}, a second application of the monotone converge theorem to the non-negative random variables $\omega\mapsto\log\mathcal{E}_{\omega,\tau}(B_{0,\delta},\varphi)-\log\mathcal{E}_{\omega,\tau}(B_{0,\delta_\mathrm{F}},\varphi)$ entails existence of a number $\delta\ge\delta_\mathrm{F}$ such that
\begin{equation}
\label{upper_cumulant_1}
\zeta\le \Ex\bigg[\frac{1}{\tau}\log\mathcal{E}_{\bdot,\tau}(B_{0,\delta},\varphi)\bigg]+2\epsilon.
\end{equation}
The crucial point is that we can replace the open ball $B_{0,\delta}$ by a compact set, according to the following lemma.

\begin{lemma}
\label{lem:ball_to_compact}
For every $\tau\ge t_{\mathrm{F}}$, $\delta\ge\delta_\mathrm{F}$, $\varphi\in\rew^\star$, and $\epsilon>0$, there exists a compact set $K\subset\rew$ such that $\Ex[\log\mathcal{E}_{\bdot,\tau}(B_{0,\delta},\varphi)]\le\Ex[\log\mathcal{E}_{\bdot,\tau}(K,\varphi)]+\epsilon$.
\end{lemma}

\begin{proof}[Proof of Lemma \ref{lem:ball_to_compact}]
Fix $\tau\ge t_{\mathrm{F}}$, $\delta\ge\delta_\mathrm{F}$, $\varphi\in\rew^\star$, and $\epsilon>0$. Recall that there exists a $K_o\subset\rew$ compact such that $\Ex[\log\mathcal{E}_{\bdot,\tau}(K_o,\varphi)]$ exists and is finite, as seen in Section \ref{sec:super}. Let $\delta_o\ge \delta$ be a number that satisfies $B_{0,\delta_o}\supset K_o$ and denote by $F_o$ the closure of $B_{0,\delta_o}$ for brevity. Clearly, $B_{0,\delta}\subseteq B_{0,\delta_o}\subset F_o\subset B_{0,2\delta_o}$ and $K_o\subset F_o$. Since $\min\{0,\log\mathcal{E}_{\omega,\tau}(F_o,\varphi)\}\ge\min\{0,\log\mathcal{E}_{\omega,\tau}(B_{0,\delta_o},\varphi)\}$ and $\max\{0,\log\mathcal{E}_{\omega,\tau}(F_o,\varphi)\}\le \max\{0,\log\mathcal{E}_{\omega,\tau}(B_{0,2\delta_o},\varphi)\}$, the expectation $\Ex[\log\mathcal{E}_{\bdot,\tau}(F_o,\varphi)]$ exists and is finite by Lemma \ref{integrability_Z}. The present lemma is demonstrated once we show that
\begin{equation}
\Ex\big[\log\mathcal{E}_{\bdot,\tau}(F_o,\varphi)\big]\le\Ex\big[\log\mathcal{E}_{\bdot,\tau}(K,\varphi)\big]+\epsilon
\label{lem:ball_to_compact_1}
\end{equation}
for some compact set $K\subset\rew$.

For $A\in\Brew$ and $i\in\{1,\ldots,\tau\}$, put 
\begin{equation*}
\Delta_{\omega,i}(A):=E_\omega\Big[\mathds{1}_{\{Y_{\omega,1}\in F_o,\ldots,Y_{\omega,N_\tau}
\in F_o,\,Y_{\omega,i}\in A,\,N_\tau\ge i,\,\tau\in\mathcal{T}\}}
\ee^{\varphi(W_\tau)+H_{\omega,\tau}}\Big]\le \mathcal{E}_{\omega,\tau}(F_o,\varphi).
\end{equation*}
We note that the measure that associates $A\in\Brew$ with
\begin{equation*}
\Ex\Bigg[\bigg\{1+\log\frac{\mathcal{E}_{\bdot,\tau}(F_o,\varphi)}{\mathcal{E}_{\bdot,\tau}(K_o,\varphi)}\bigg\}
\frac{\Delta_{\bdot,i}(A)}{\mathcal{E}_{\bdot,\tau}(F_o,\varphi)}\Bigg]\le
\Ex\bigg[1+\log\frac{\mathcal{E}_{\bdot,\tau}(F_o,\varphi)}{\mathcal{E}_{\bdot,\tau}(K_o,\varphi)}\bigg]
\end{equation*}
is finite. Therefore, due to the separability of $\rew$, such a measure is tight (see \cite{bogachevbook}, Theorem 7.1.7), so that there exists a $K_i\subset\rew$ compact with the property that
\begin{equation}
\Ex\Bigg[\bigg\{1+\log\frac{\mathcal{E}_{\bdot,\tau}(F_o,\varphi)}{\mathcal{E}_{\bdot,\tau}(K_o,\varphi)}\bigg\}
\frac{\Delta_{\bdot,i}(K_i^c)}{\mathcal{E}_{\bdot,\tau}(F_o,\varphi)}\Bigg]\le\frac{\epsilon}{\tau}.
\label{lem:ball_to_compact_2}
\end{equation}
Define $K:=(K_o\cup K_1\cup\cdots\cup K_\tau)\cap F_o$. We claim that $K$ is the desired compact set. In fact, we resort to the bound $\log(1+\zeta)\le\{1+\log(1+\zeta)\}\frac{\zeta}{1+\zeta}$ valid for $\zeta\ge 0$, to write down
\begin{align}
\nonumber
\log\mathcal{E}_{\omega,\tau}(F_o,\varphi)&=\log\mathcal{E}_{\omega,\tau}(K,\varphi)
+ \log\bigg\{1+\frac{\mathcal{E}_{\omega,\tau}(F_o,\varphi)
-\mathcal{E}_{\omega,\tau}(K,\varphi)}{\mathcal{E}_{\omega,\tau}(K,\varphi)}\bigg\}\\
\nonumber
&\le \log\mathcal{E}_{\omega,\tau}(K,\varphi)
+\bigg\{1+\log\frac{\mathcal{E}_{\omega,\tau}(F_o,\varphi)}{\mathcal{E}_{\omega,\tau}(K,\varphi)}\bigg\}
\frac{\mathcal{E}_{\omega,\tau}(F_o,\varphi)
-\mathcal{E}_{\omega,\tau}(K,\varphi)}{\mathcal{E}_{\omega,\tau}(F_o,\varphi)}.
\end{align}
Since $\mathcal{E}_{\omega,\tau}(F_o,\varphi)\ge\mathcal{E}_{\omega,\tau}(K,\varphi)\ge\mathcal{E}_{\omega,\tau}(K_o,\varphi)$, we can even replace $\mathcal{E}_{\omega,\tau}(K,\varphi)$ by $\mathcal{E}_{\omega,\tau}(K_o,\varphi)$ in the second logarithm to obtain
\begin{equation}  
\log\mathcal{E}_{\omega,\tau}(F_o,\varphi)\le\log\mathcal{E}_{\omega,\tau}(K,\varphi)
+\bigg\{1+\log\frac{\mathcal{E}_{\omega,\tau}(F_o,\varphi)}{\mathcal{E}_{\omega,\tau}(K_o,\varphi)}\bigg\}
\frac{\mathcal{E}_{\omega,\tau}(F_o,\varphi)-\mathcal{E}_{\omega,\tau}(K,\varphi)}{\mathcal{E}_{\omega,\tau}(F_o,\varphi)}.
\label{lem:ball_to_compact_3}
\end{equation}
Next, we observe that either $Y_{\omega,i}\in K$ for $i=1,\ldots,N_\tau$ or there exists a $i\le N_\tau$ such that $Y_{\omega,i}\in K^c$. We also observe that $\Delta_{\omega,i}(K^c)=\Delta_{\omega,i}(K^c\cap F_o)\le\Delta_{\omega,i}(K_i^c)$, as $\Delta_{\omega,i}(A)=\Delta_{\omega,i}(A\cap F_o)$ for any $A\in\Brew$. These arguments give
\begin{equation*}
\mathcal{E}_{\omega,\tau}(F_o,\varphi)\le \mathcal{E}_{\omega,\tau}(K,\varphi)+\sum_{i=1}^\tau\Delta_{\omega,i}(K^c)
\le \mathcal{E}_{\omega,\tau}(K,\varphi)+\sum_{i=1}^\tau\Delta_{\omega,i}(K_i^c).
\end{equation*}
Combining this inequality with (\ref{lem:ball_to_compact_3}), we get the bound
\begin{equation*}
\log\mathcal{E}_{\omega,\tau}(F_o,\varphi)\le \log\mathcal{E}_{\omega,\tau}(K,\varphi)
+\sum_{i=1}^\tau\bigg\{1+\log\frac{\mathcal{E}_{\omega,\tau}(F_o,\varphi)}{\mathcal{E}_{\omega,\tau}(K_o,\varphi)}\bigg\}
\frac{\Delta_{\omega,i}(K_i^c)}{\mathcal{E}_{\omega,\tau}(F_o,\varphi)}.
\end{equation*}
After taking expectation, this bound becomes (\ref{lem:ball_to_compact_1}) thanks to (\ref{lem:ball_to_compact_2}).
\end{proof}

By the bound in (\ref{upper_cumulant_1}) and Lemma \ref{lem:ball_to_compact}, there exists a compact set $K\subset\rew$ such that
\begin{equation*}
\zeta\le \Ex\bigg[\frac{1}{\tau}\log\mathcal{E}_{\bdot,\tau}(K,\varphi)\bigg]+3\epsilon.
\end{equation*}
Replacing $K$ by the closed convex hull of $\{0\}\cup K$ if necessary, we can treat $K$ as a compact convex set that contains $0$. For every $n\in\N$, Lemma \ref{super_Z} implies that $\Ex[\frac{1}{\tau}\log\mathcal{E}_{\bdot,\tau}(K,\varphi)]\le\Ex[\frac{1}{n\tau}\log\mathcal{E}_{\bdot,n\tau}(K,\varphi)]$, so that
\begin{equation*}
\zeta\le \Ex\bigg[\frac{1}{n\tau}\log\mathcal{E}_{\bdot,n\tau}(K,\varphi)\bigg]+3\epsilon.
\end{equation*}
If $\rho\ge\delta_\mathrm{F}$ is a number such that $K\subset B_{0,\rho}$, then, defining $e:=\sup_{t\ge t_\mathrm{F}}\{\Ex[\frac{1}{t}\log\mathcal{E}_{\bdot,t}(B_{0,\rho},\varphi)]\}$, we can even conclude that
\begin{equation*}
\zeta\le \Ex\bigg[\frac{1}{n\tau}\log\mathcal{E}_{\bdot,n\tau}(K,\varphi)-\frac{1}{n\tau}
\log\mathcal{E}_{\bdot,n\tau}(B_{0,\rho},\varphi)\bigg]+e+3\epsilon.
\end{equation*}
In this way, as the random variables $\omega\mapsto \log\mathcal{E}_{\omega,n\tau}(B_{0,\rho},\varphi)-\log\mathcal{E}_{\omega,n\tau}(K,\varphi)$
are non-negative, Fatou's lemma gives
\begin{align}
\nonumber
\zeta 
&\le \limsup_{n\uparrow\infty}\Ex\bigg[\frac{1}{n\tau}\log\mathcal{E}_{\bdot,n\tau}(K,\varphi)
-\frac{1}{n\tau}\log\mathcal{E}_{\bdot,n\tau}(B_{0,\rho},\varphi)\bigg]+e+3\epsilon\\
\nonumber
&\le \Ex\bigg[\limsup_{t\uparrow\infty}\bigg\{\frac{1}{t}\log\mathcal{E}_{\bdot,t}(K,\varphi)
-\frac{1}{t}\log\mathcal{E}_{\bdot,t}(B_{0,\rho},\varphi)\bigg\}\bigg]+e+3\epsilon.
\end{align}
Since $\lim_{t\uparrow\infty}\frac{1}{t}\log\mathcal{E}_{\omega,t}(B_{0,\rho},\varphi)=e$ $\pae$ by Corollary \ref{Z_truncated}, from here we get the desired bound $\zeta\le J_o^\star(\varphi)+3\epsilon$ once we prove that
\begin{equation}
\label{eq:ultima_zJ}
\limsup_{t\uparrow\infty}\frac{1}{t}\log\mathcal{E}_{\omega,t}(K,\varphi)\le J_o^\star(\varphi) \qquad \pae.
\end{equation}
To this aim, we note that the conditions $Y_{\omega,1}\in K,\ldots,Y_{\omega,N_t}\in K$ entail $\frac{1}{t}\sum_{i=1}^{N_t}Y_{\omega,i}=(1-\frac{N_t}{t})0+\frac{1}{t}\sum_{i=1}^{N_t}Y_{\omega,i}\in K$, since $K$ is convex with $0\in K$. It follows that, for $t\in\N$,
\begin{align}
\nonumber
\mathcal{E}_{\omega,t}(K,\varphi)
&:=E_\omega\Big[\mathds{1}_{\{Y_{\omega,1}\in K,\ldots,Y_{\omega,N_t}\in K,\,t\in\mathcal{T}\}}\ee^{\varphi(W_t)+H_{\omega,t}}\Big]\\
&\le E_\omega\bigg[\mathds{1}_{\big\{\frac{1}{t}\sum_{i=1}^{N_t}Y_{\omega,i}\in K,\,t\in\mathcal{T}\big\}}
\ee^{\varphi(W_t)+H_{\omega,t}}\bigg].
\label{eq:ultima_zJ_1}
\end{align}
The following lemma concludes the proof because it yields (\ref{eq:ultima_zJ}) through (\ref{eq:ultima_zJ_1}).

\begin{lemma}
\label{lemma:aux}
For any compact set $K\in\rew$ and $\varphi\in\rew^\star$,
\begin{equation*}
\limsup_{t\uparrow\infty}\frac{1}{t}\log E_{\omega}\bigg[\mathds{1}_{\big\{\frac{1}{t}\sum_{i=1}^{N_t}
Y_{\omega,i}\in K,\,t\in\mathcal{T}\big\}}\ee^{\varphi(W_t)+H_{\omega,t}}\bigg] \le J_o^\star(\varphi) \qquad \pae.
\end{equation*}
\end{lemma}

\begin{proof}[Proof of Lemma~\ref{lemma:aux}]
Fix $K\subset\rew$ compact and $\varphi\in\rew^\star$. Denote by $F$ the closed ball of center $0$ and radius $2\,\Ex[\sup_{s\in\N}\max\{0,\|r_{\bdot,s}\|-\eta s\}]+2\eta<+\infty$, with $\eta$ the number in Assumption \ref{ass2}. We will prove that $\pae$
\begin{equation}
\label{lemma:aux_prima}
\limsup_{t\uparrow\infty}\frac{1}{t}\log E_\omega\bigg[\mathds{1}_{\big\{\frac{W_t}{t}\in \mathcal{V}\cap F+K,\,t\in\mathcal{T}\big\}}
\ee^{\varphi(W_t)+H_{\omega,t}}\bigg] \le J_o^\star(\varphi),
\end{equation}
where $\mathcal{V}$ is the finite-dimensional subspace in Assumption \ref{ass2}. This bound gives us the lemma as follows. Since $\sum_{i=1}^{N_t}r_{f^{T_{i-1}}\omega,S_i}\in\mathcal{V}$ and
\begin{align}
\nonumber
\bigg\|\sum_{i=1}^{N_t}r_{f^{T_{i-1}}\omega,S_i}\bigg\|&\le\sum_{i=1}^{N_t}\max\big\{0,\|r_{f^{T_{i-1}}\omega,S_i}\|
-\eta S_i\big\}+\eta T_{N_t}\\
\nonumber
&\le\sum_{\tau=0}^{t-1}\sup_{s\in\N}\max\big\{0,\|r_{f^\tau\omega,s}\|-\eta s\big\}+\eta t,
\end{align}
Birkhoff's ergodic theorem entails that $\pae$ the vector $\frac{1}{t}\sum_{i=1}^{N_t}r_{f^{T_{i-1}}\omega,S_i}$ belongs to $\mathcal{V}\cap F$ for all sufficiently large $t$. Thus, $\pae$ the condition $\frac{W_t}{t}-\frac{1}{t}\sum_{i=1}^{N_t}r_{f^{T_{i-1}}\omega,S_i}=\frac{1}{t}\sum_{i=1}^{N_t}Y_{\omega,i}\in K$ implies that $\frac{W_t}{t}\in \mathcal{V}\cap F+K$ for all sufficiently large $t$.

The bound in (\ref{lemma:aux_prima}) relies on the compactness of $\mathcal{V}\cap F+K$, for which the dimension of $\mathcal{V}$ must
be finite. In order to demonstrate this bound, recall that, by $(ii)$ of Proposition \ref{prop:ldp_Jo}, there exists an $\Omega_o\in\mathcal{F}$
with $\prob[\Omega_o]=1$ such that
\begin{equation*}
\inf_{\delta>0}\bigg\{\limsup_{t\uparrow\infty} \frac{1}{t}\log\mu_{\omega,t}(B_{w,\delta})\bigg\}\le-J_o(w) 
\end{equation*}
for all $w\in\rew$ and $\omega\in\Omega_o$. Fix $\omega\in\Omega_o$ and $\epsilon>0$. Then, for each $w$ there exists $\delta_w>0$ such that $\limsup_{t\uparrow\infty}\frac{1}{t}\log\mu_{\omega,t}(B_{w,\delta_w})\le -J_o(w)+\epsilon$. The number $\delta_w$ can be chosen so small that $\|\varphi\|\delta_w\le\epsilon$. By compactness, there exist finitely many points $w_1,\ldots,w_n$ in $\mathcal{V}\cap F+K$ such that $\mathcal{V}\cap F+K\subseteq\cup_{i=1}^nB_{w_i,\delta_{w_i}}$. This gives
\begin{align}
\nonumber
E_\omega\bigg[\mathds{1}_{\big\{\frac{W_t}{t}\in \mathcal{V}\cap F+K,\,t\in\mathcal{T}\big\}}\ee^{\varphi(W_t)+H_{\omega,t}}\bigg]
&\le \sum_{i=1}^nE_\omega\bigg[\mathds{1}_{\big\{\frac{W_t}{t}\in B_{w_i,\delta_{w_i}},\,t\in\mathcal{T}\big\}}
\ee^{\varphi(W_t)+H_{\omega,t}}\bigg]\\
\nonumber
&\le\sum_{i=1}^n \ee^{t\varphi(w_i)+t\|\varphi\|\delta_{w_i}}\mu_{\omega,t}(B_{w_i,\delta_{w_i}})\\
\nonumber
&\le\sum_{i=1}^n \ee^{t\varphi(w_i)+t\epsilon}\mu_{\omega,t}(B_{w_i,\delta_{w_i}}),
\end{align}
and it therefore follows that 
\begin{align}
\nonumber
\limsup_{t\uparrow\infty}&\frac{1}{t}\log  E_\omega\bigg[\mathds{1}_{\big\{\frac{W_t}{t}\in \mathcal{V}\cap F+K,\,t\in\mathcal{T}\big\}}
\ee^{\varphi(W_t)+H_{\omega,t}}\bigg]\\
\nonumber
&\le\max\Big\{\varphi(w_1)-J_o(w_1),\ldots,\varphi(w_n)-J_o(w_n)\Big\}+2\epsilon\le J_o^\star(\varphi)+2\epsilon.
\end{align}
The arbitrariness of $\epsilon$ demonstrates (\ref{lemma:aux_prima}).
\end{proof}


\subsection{The full LDP}
\label{sec:FLDPc} 

It is well known that a weak LDP gives a full LDP with good rate function when it is combined with exponential tightness (see \cite{dembobook}, Lemma 1.2.18). Hence Corollary \ref{FLDPc} is a consequence of the following lemma.

\begin{lemma}
\label{lem:exp_tight}
Under the hypotheses of Corollary \ref{FLDPc}, $\pae$ the family of measures $\{\mu_{\omega,t}\}_{t\in\N}$ is exponentially tight, i.e., for each number $\lambda>0$ there exists a compact set $K\subset\rew$ such that
\begin{equation*}
\limsup_{t\uparrow\infty}\frac{1}{t}\log\mu_{\omega,t}(K^c)\le -\lambda.
\end{equation*}
\end{lemma}

\begin{proof}[Proof of Lemma \ref{lem:exp_tight}]
Pick a number $\rho>0$, and let $\xi>0$ be as in Corollary \ref{FLDPc}. For $t\in\N$, a Chernoff-type bound, together with the bound in (\ref{H_bound}), gives
\begin{align}
\nonumber
\mu_{\omega,t}(B_{0,\rho}^c)
&:= E_\omega\bigg[\mathds{1}_{\big\{\frac{W_t}{t}\in B_{0,\rho}^c,\,t\in\mathcal{T}\big\}} \ee^{H_{\omega,t}}\bigg]\\
\nonumber
&\le E_\omega\Big[\mathds{1}_{\{\sum_{i=1}^{N_t}\|X_i\|\ge \rho t,\,t\in\mathcal{T}\}} \ee^{H_{\omega,t}}\Big]\\
\nonumber
&\le \ee^{-\xi\rho t}E_\omega\Big[\mathds{1}_{\{t\in\mathcal{T}\}} \ee^{\xi\sum_{i=1}^{N_t}\|X_i\|+H_{\omega,t}}\Big]\\
&\le\ee^{\sum_{\tau=0}^{t-1}\sup_{s\in\N}\max\{0,v_{f^\tau\omega}(s)-\eta s\}+\eta t-\xi\rho t}
E_\omega\Big[\mathds{1}_{\{t\in\mathcal{T}\}}\ee^{\xi \sum_{i=1}^{N_t}\|X_i\|}\Big].
\label{lem:exp_tight_1}
\end{align}
At the same time, with $M$ as in Corollary \ref{FLDPc}, we have
\begin{align}
\nonumber
E_\omega
&\Big[\mathds{1}_{\{t\in\mathcal{T}\}}\ee^{\xi \sum_{i=1}^{N_t}\|X_i\|}\Big]\\
\nonumber
&=\sum_{n=1}^t\sum_{s_1\in\N}\cdots\sum_{s_n\in\N}\mathds{1}_{\{t_n=t\}}\prod_{i=1}^n
p_{f^{t_{i-1}}\omega}(s_i)\int_\rew \ee^{\xi\|x\|}\lambda_{f^{t_{i-1}}\omega}(\dd x|s_i)\\
\nonumber
&=\ee^{Mt}\sum_{n=1}^t\sum_{s_1\in\N}\cdots\sum_{s_n\in\N}\mathds{1}_{\{t_n=t\}}\prod_{i=1}^n
p_{f^{t_{i-1}}\omega}(s_i)\int_\rew \ee^{\xi\|x\|-Ms_i}\lambda_{f^{t_{i-1}}\omega}(\dd x|s_i)\\
\nonumber
&\le\ee^{Mt}\sum_{n=1}^t\sum_{s_1\in\N}\cdots\sum_{s_n\in\N}\mathds{1}_{\{t_n=t\}}\prod_{i=1}^n
p_{f^{t_{i-1}}\omega}(s_i)\,\ee^{\sup_{s\in\N}\max\big\{0,\log\int_\rew \ee^{\xi\|x\|-Ms}\lambda_{f^{t_{i-1}}\omega}(\dd x|s)\big\}}\\
\nonumber
&\le\ee^{Mt+\sum_{\tau=0}^{t-1}\sup_{s\in\N}\max\big\{0,\log\int_\rew \ee^{\xi\|x\|-Ms}\lambda_{f^\tau\omega}(\dd x|s)\big\}} P_\omega[t\in\mathcal{T}]\\
&\le\ee^{Mt+\sum_{\tau=0}^{t-1}\sup_{s\in\N}\max\big\{0,\log\int_\rew \ee^{\xi\|x\|-Ms}\lambda_{f^\tau\omega}(\dd x|s)\big\}},
\label{lem:exp_tight_2}
\end{align}
where $t_0:=0$ and $t_i:=s_1+\cdots+s_i$ for $i\in\N$. Since $\Ex[\sup_{s\in\N}\max\{0, v_{\bdot}(s)-\eta s\}]<+\infty$ by Assumption \ref{ass3} and
$\Ex[\sup_{s\in\N}\max\{0,\log\int_\rew\ee^{\xi\|x\|-Ms}\lambda_{\bdot}(\dd x|s)\}]<+\infty$ by hypothesis, combining (\ref{lem:exp_tight_1}) with (\ref{lem:exp_tight_2}), taking logarithm, dividing by $t$, and letting $t\uparrow\infty$, we find 
\begin{align}
\nonumber
\limsup_{t\uparrow\infty}\frac{1}{t}\log\mu_{\omega,t}(B_{0,\rho}^c)&\le\Ex\Big[\sup_{s\in\N}\max\big\{0,v_{\bdot}(s)-\eta s\big\}\Big]
+\eta-\xi\rho\\
\nonumber
&\qquad +M+\Ex\bigg[\sup_{s\in\N}\max\bigg\{0,\log\int_\rew\ee^{\xi\|x\|-Ms}\lambda_{\bdot}(\dd x|s)\bigg\}\bigg] \qquad \pae
\end{align}
thanks to Birkhoff's ergodic theorem. In this way, given $\lambda>0$, the lemma is proved by choosing $\rho$ so large that the r.h.s.\ is
smaller than $\lambda$ and $K$ equal to the closure of $B_{0,\rho}$. 
\end{proof}


\section{Free LDPs}
\label{sec:freemodels}

In this section we deduce Theorem \ref{WLDPfree} and Corollary \ref{FLDPfree} from our results on the constrained pinning model. To this aim, we note that, for $t\in\N_0$ and $0\le n\le\tau\le t$, the disjoint events $\{T_n=\tau,\,T_{n+1}>t\}=\{T_n=\tau,\,S_{n+1}>t-\tau\}$ exhaust the entire space of configurations. According to (\ref{start_0}), the conditions $T_n=\tau$ and $S_{n+1}>t-\tau$ entail that $H_{\omega,t}=\sum_{i=1}^nv_{f^{T_{i-1}}\omega}(S_i)=H_{\omega,\tau}$ and $W_t=\sum_{i=1}^nX_i=W_\tau$ are independent of $S_{n+1}$, which in turn is distributed as $S_1$ in the environment $f^\tau\omega$. In this way, we find the identity of measures
\begin{align}
\nonumber
\nu_{\omega,t}:=E_\omega\bigg[\mathds{1}_{\big\{\frac{W_t}{t}\in \,\bdot\,\big\}} \ee^{H_{\omega,t}}\bigg]
&=\sum_{\tau=0}^t\sum_{n=0}^\tau E_\omega\bigg[\mathds{1}_{\big\{\frac{W_t}{t}\in \,\bdot\,,\,T_n=\tau,\,
S_{n+1}>t-\tau\big\}} \ee^{H_{\omega,t}}\bigg]\\
\nonumber
&=\sum_{\tau=0}^t\sum_{n=0}^\tau E_\omega\bigg[\mathds{1}_{\big\{\frac{W_\tau}{t}\in \,\bdot\,,\,T_n=\tau\big\}} 
\ee^{H_{\omega,\tau}}\bigg] P_{f^\tau\omega}[S_1>t-\tau]\\
&=\sum_{\tau=0}^tE_\omega\bigg[\mathds{1}_{\big\{\frac{W_\tau}{t}\in \,\bdot\,,\,\tau\in\mathcal{T}\big\}} \ee^{H_{\omega,\tau}}\bigg] P_{f^\tau\omega}[S_1>t-\tau].
\label{mu2nu}
\end{align}
Similarly, for $\varphi\in\rew^\star$, we have
\begin{align}
\nonumber
\int_\rew\ee^{t\varphi(w)}\nu_{\omega,t}(\dd w)=E_\omega\Big[\ee^{\varphi(W_t)+H_{\omega,t}}\Big]
&=\sum_{\tau=0}^tE_\omega\Big[\mathds{1}_{\{\tau\in\mathcal{T}\}} \ee^{\varphi(W_\tau)+H_{\omega,\tau}}\Big] P_{f^\tau\omega}[S_1>t-\tau]\\
&=\sum_{\tau=0}^t Z_{\omega,\tau}(\varphi)P_{f^\tau\omega}[S_1>t-\tau].
\label{CGF_free}
\end{align}
These formulas connect the free setting with the constrained setting, and represent the starting point to prove Theorem \ref{WLDPfree} and
Corollary \ref{FLDPfree}.

From now on we assume that \eqref{Ptailcond} holds. We define the rate function $I_\ell$ as the Legendre transform of $z_\ell:=\max\{z,\ell\}$. Recall that $z=z_o$ is proper convex and lower semi-continuous. The identity in (\ref{CGF_free}) immediately gives $(i)$ of Theorem \ref{WLDPfree} according to the following lemma.

\begin{lemma}
\label{lem:CGF_free}
For $\varphi\in\rew^\star$, $\lim_{t\uparrow\infty}\frac{1}{t}\log\int_\rew\ee^{t\varphi(w)}\nu_{\omega,t}(\dd w)=z_\ell(\varphi)$ $\pae$.
\end{lemma}

\begin{proof}[Proof of Lemma \ref{lem:CGF_free}]
Fix $\varphi\in\rew^\star$. Corollary \ref{Zc} states that there exists a set $\Omega_o\in\mathcal{F}$ with $\prob[\Omega_o]=1$ such that $\lim_{t\uparrow\infty}\frac{1}{t}\log Z_{\omega,t}(\varphi)=z_o(\varphi)=z(\varphi)$ for all $\omega\in\Omega_o$. We can choose $\Omega_o$ so that also \eqref{Ptailcond} is satisfied for all $\omega\in\Omega_o$. Pick $\omega\in\Omega_o$. Since $\int_\rew\ee^{t\varphi(w)}\nu_{\omega,t}(\dd w)\ge \max\{Z_{\omega,t}(\varphi),P_\omega[S_1>t]\}$ by (\ref{CGF_free}) and $\liminf_{s\uparrow\infty}\frac{1}{s}\log P_\omega[S_1>s]\ge\ell$ by
the second line of \eqref{Ptailcond}, we can conclude that
\begin{equation*}
\liminf_{t\uparrow\infty}\frac{1}{t}\log\int_\rew\ee^{t\varphi(w)}\nu_{\omega,t}(\dd w)\ge\max\{z(\varphi),\ell\}=:z_\ell(\varphi).
\end{equation*}
This settles the first half of the proof and already proves the lemma when $z_\ell(\varphi)=+\infty$. If $z_\ell(\varphi)<+\infty$, then, given real numbers $\lambda>z_\ell(\varphi)$ and $\epsilon>0$, there exists a positive constant $C$ such that $Z_{\omega,\tau}(\varphi)\le C\ee^{\lambda \tau}$ and $P_{f^\tau\omega}[S_1>s]\le C\ee^{\lambda s+\epsilon\tau}$ for all $s,\tau\in\N_0$, the latter being a consequence of the first line of \eqref{Ptailcond}. Combining (\ref{CGF_free}) with these bounds, we find $\limsup_{t\uparrow\infty}\frac{1}{t}\log\int_\rew\ee^{t\varphi(w)}\nu_{\omega,t}(\dd w)\le\lambda+\epsilon$. The arbitrariness of $\lambda$ and $\epsilon$ shows that
\begin{equation*}
\limsup_{t\uparrow\infty}\frac{1}{t}\log\int_\rew\ee^{t\varphi(w)}\nu_{\omega,t}(\dd w)\le z_\ell(\varphi).
\qedhere
\end{equation*}
\end{proof}

The proof that the family $\{\nu_{\omega,t}\}_{t \in \N}$ satisfies a quenched weak LDP with the rate function $I_\ell$ is given in Section \ref{sec:weak_nu_ellinf} for the case $\ell=-\infty$ and in Section \ref{sec:weak_nu_ellfin} for the case $\ell>-\infty$. This verifies $(ii)$ of Theorem \ref{WLDPfree}. Section \ref{sec:FLDPfree} addresses the quenched full LDP of Corollary \ref{FLDPfree}.


\subsection{The weak LDP for infinite exponential tail constant}
\label{sec:weak_nu_ellinf}

When $\ell=-\infty$, we have $I_\ell(w)=\sup_{\varphi\in\rew^\star}\{\varphi(w)-z(\varphi)\}=J(w)=J_o(w)$. Thus, $(ii)$ of Theorem \ref{WLDPfree} for $\ell=-\infty$ follows from the following proposition.

\begin{proposition}
\label{prop:lower_ell_inf}
If $\ell=-\infty$, then $\pae$ the family $\{\nu_{\omega,t}\}_{t \in \N}$ satisfies the weak LDP with rate function $J_o$.
\end{proposition}

\begin{proof}[Proof of Proposition \ref{prop:lower_ell_inf}]
The large deviation lower bound for open sets is immediate from Proposition \ref{prop:ldp_Jo}, as $\nu_{\omega,t}(A)\ge\mu_{\omega,t}(A)$ for all $t\in\N$ and $A\in\Brew$. We verify the large deviation upper bound for compact sets.

According to Corollary \ref{Zc}, $(ii)$ of Proposition \ref{prop:ldp_Jo}, and the first line of \eqref{Ptailcond} with $\ell=-\infty$, there exists a set $\Omega_o\in\mathcal{F}$ with $\prob[\Omega_o]=1$ such that, for every $\omega\in\Omega_o$, we have $\lim_{t\uparrow\infty}\frac{1}{t}\log Z_{\omega,t}(0)=z_o(0)$ with $z_o(0)$ finite, $\inf_{\delta>0}\{\limsup_{t\uparrow\infty}\frac{1}{t}\log\mu_{\omega,t}(B_{w,\delta})\}\le -J_o(w)$ for $w\in\rew$, and
\begin{equation}
\label{tail_prob}
\limsup_{s\uparrow\infty}\sup_{\tau\in\N_0}\bigg\{\frac{1}{s}\log P_{f^\tau\omega}[S_1>s]-\frac{\tau}{s}\bigg\}=-\infty.
\end{equation}
We will show that, for all $\omega\in\Omega_o$ and $w\in\rew$,
\begin{equation}
\inf_{\delta>0}\bigg\{\limsup_{t\uparrow\infty}\frac{1}{t}\log\nu_{\omega,t}(B_{w,\delta})\bigg\}\le-J_o(w).
\label{prop:lower_ell_inf_1}
\end{equation}
This gives the quenched large deviation upper bound for compact sets with rate function $J_o$, as in the proof of Proposition \ref{prop:ldp_Jo}.

Pick $\omega\in\Omega_o$, $w\in\rew$, and a real number $\lambda<J_o(w)$. Then there exists an $\eta>0$ such that $\limsup_{t\uparrow\infty}\frac{1}{t}\log\mu_{\omega,t}(B_{w,4\eta})<-\lambda$. In turn, there exists a positive constant $C$ that provides the inequalities $\mu_{\omega,\tau}(B_{w,4\eta})\le C\ee^{-\lambda \tau}$ and $Z_{\omega,\tau}(0)\le C\ee^{z_o(0)\tau+\tau}$ for all $\tau\in\N_0$. Given $\epsilon\in(0,1/2)$ such that
$\epsilon\|w\|\le\eta$, the conditions $t/2\le(1-\epsilon)t<\tau\le t$ and $\frac{W_\tau}{t}\in B_{w,\eta}$ imply $\frac{W_\tau}{\tau}\in B_{w,4\eta}$, since $\|W_\tau-w\tau\|\le\|W_\tau-wt\|+(t-\tau)\|w\|<\eta t+\epsilon\|w\|t\le 2\eta t<4\eta\tau$. Hence it follows from (\ref{mu2nu}) that
\begin{align}
\nonumber
\nu_{\omega,t}(B_{w,\eta})&=\sum_{\tau=0}^tE_\omega\bigg[\mathds{1}_{\big\{\frac{W_\tau}{t}\in B_{w,\eta},\,\tau\in\mathcal{T}\big\}} 
\ee^{H_{\omega,\tau}}\bigg] P_{f^\tau\omega}[S_1>t-\tau]\\
\nonumber
&\le\sum_{\tau=0}^t\mathds{1}_{\{\tau\le(1-\epsilon)t\}}Z_{\omega,\tau}(0) P_{f^\tau\omega}[S_1>t-\tau]
+\sum_{\tau=0}^t\mathds{1}_{\{\tau>(1-\epsilon)t\}}\mu_{\omega,\tau}(B_{w,4\eta})\\
\nonumber
&\le tC\ee^{|z_o(0)|t+2t} \sup_{\tau\in\N_0}\big\{P_{f^\tau\omega}[S_1>\epsilon t]\ee^{-\tau}\big\}
+tC\ee^{-\lambda t+\epsilon|\lambda|t}.
\end{align}
Taking the logarithm, dividing by $t$, letting $t\uparrow\infty$, and invoking (\ref{tail_prob}), we get $\limsup_{t\uparrow\infty}\frac{1}{t}\log\nu_{\omega,t}(B_{w,\eta})$ $\le-\lambda+\epsilon|\lambda|$. Letting $\epsilon\downarrow 0$, this gives $\inf_{\delta>0}\{\limsup_{t\uparrow\infty}\frac{1}{t}\log\nu_{\omega,t}(B_{w,\delta})\}\le-\lambda$, which proves \eqref{prop:lower_ell_inf_1} thanks to the arbitrariness of $\lambda$.
\end{proof}


\subsection{The weak LDP for finite exponential tail constant}
\label{sec:weak_nu_ellfin}

The following lemma shows that the family $\{\nu_{\omega,t}\}_{t \in \N}$ satisfies the quenched large deviation lower bound for open sets
with rate function $I_\ell$ even when $\ell>-\infty$.

\begin{lemma}
\label{lem:lower_bound_free}
Assume that $\ell>-\infty$. Then $\pae$ the family $\{\nu_{\omega,t}\}_{t \in \N}$ satisfies the large deviation lower bound with rate function $I_\ell$, i.e., $\liminf_{t\uparrow\infty} \frac{1}{t}\log\nu_{\omega,t}(G) \ge-\inf_{w\in G} I_\ell(w)$ for all $G\subseteq\rew$ open.
\end{lemma}

\begin{proof}[Proof of Lemma \ref{lem:lower_bound_free}]
According to $(i)$ of Proposition \ref{prop:ldp_Jo}, there exists an event $\Omega_o\in\mathcal{F}$ with $\prob[\Omega_o]=1$ such that $\liminf_{t\uparrow\infty}\frac{1}{t}\log\mu_{\omega,t}(B_{w,\delta})\ge -J_o(w)$ for all $\omega\in\Omega_o$, $w\in\rew$, and $\delta>0$. By the second line of \eqref{Ptailcond}, we can even suppose that $\Omega_o$ satisfies $\liminf_{s\uparrow\infty}\inf_{\tau\in\N_0}\{\frac{1}{s}\log P_{f^\tau\omega}[S_1>s]+\epsilon \frac{\tau}{s}\}\ge\ell$ for all $\omega\in\Omega_o$ with any number $\epsilon>0$.  Pick $\omega\in\Omega_o$.  We verify that
\begin{equation}
\liminf_{t\uparrow\infty}\frac{1}{t}\log \nu_{\omega,t}(B_{w,\delta})\ge-I_\ell(w)
\label{lower_free_1}
\end{equation}
for all $w\in\rew$ and $\delta>0$, which implies the quenched large deviation lower bound for open sets, as we have seen in the proof of
Proposition \ref{prop:ldp_Jo}.

Fix $w\in\rew$ and $\delta>0$. We will prove that
\begin{equation}
\liminf_{t\uparrow\infty}\frac{1}{t}\log \nu_{\omega,t}(B_{w,\delta})\ge
-\sup_{\varphi\in\scriptsize{\dom} z}\Big\{\varphi(w)-\beta z(\varphi)-(1-\beta)\ell\Big\}
\label{lower_free_2}
\end{equation}
for all $\beta\in[0,1]$, where $\dom z:=\{\varphi\in\rew^\star:z(\varphi)<+\infty\}$ is the effective domain of the proper convex lower semi-continuous function $z=z_o$. This gives (\ref{lower_free_1}) as follows. The function that maps $(\beta,\varphi)\in[0,1]\times\dom z$ to the real number $\varphi(w)-\beta z(\varphi)-(1-\beta)\ell$ is concave and upper semi-continuous with respect to $\varphi$ for each fixed $\beta\in[0,1]$, and is convex and continuous with respect to $\beta$ for each fixed $\varphi\in\dom z$. Due to compactness of the closed interval $[0,1]$, Sion's minimax theorem allows us to exchange the infimum over $\beta\in[0,1]$ and the supremum over $\varphi\in\dom z$, to conclude that
\begin{align}
\nonumber
I_\ell(w):=\sup_{\varphi\in\rew^\star}\Big\{\varphi(w)-z_\ell(\varphi)\Big\}
&=\adjustlimits\sup_{\varphi\in\scriptsize{\dom}z}
\inf_{\beta\in[0,1]}\Big\{\varphi(w)-\beta z(\varphi)-(1-\beta)\ell\Big\}\\
\nonumber
&=\adjustlimits\inf_{\beta\in[0,1]}\sup_{\varphi\in\scriptsize{\dom}z}\Big\{\varphi(w)-\beta z(\varphi)-(1-\beta)\ell\Big\}.
\end{align}
In this way, (\ref{lower_free_2}) demonstrates (\ref{lower_free_1}) after we take the supremum over $\beta$.

We prove (\ref{lower_free_2}), considering the case $\beta>0$ first. Pick $\beta\in(0,1]$ and denote by $\tau_t$ the largest integer less than or equal to $\beta t$. Focus on the sufficiently large integers $t$ satisfying $\tau_t>0$ and $\|w\|<\beta\delta t/2$. Then the condition $\frac{W_{\tau_t}}{\tau_t}\in B_{w/\beta,\delta/2}$ implies $\frac{W_{\tau_t}}{t}\in B_{w,\delta}$. Indeed, observing that $0\le t-\tau_t/\beta<1/\beta$, if $\|W_{\tau_t}-\tau_tw/\beta\|<\delta \tau_t/2$, then $\|W_{\tau_t}-tw\|\le\|W_{\tau_t}-\tau_t w/\beta\|+(t-\tau_t/\beta)\|w\|<\delta\tau_t/2+\|w\|/\beta<\delta t$ as $\tau_t\le t$ and $\|w\|<\beta\delta t/2$. Thus, keeping only the term corresponding to $\tau=\tau_t>0$ in the r.h.s.\ of (\ref{mu2nu}), we obtain
\begin{align}
\nonumber
\nu_{\omega,t}(B_{w,\delta})
&\ge E_\omega\bigg[\mathds{1}_{\big\{\frac{W_{\tau_t}}{t}\in B_{w,\delta},\,\tau_t\in\mathcal{T}\big\}}\ee^{H_{\omega,\tau_t}}\bigg]
P_{f^{\tau_t}\omega}\big[S_1>t-\tau_t\big]\\
\nonumber
&\ge E_\omega\bigg[\mathds{1}_{\big\{\frac{W_{\tau_t}}{\tau_t}\in B_{w/\beta,\delta/2},\,\tau_t\in\mathcal{T}\big\}}
\ee^{H_{\omega,\tau_t}}\bigg]P_{f^{\tau_t}\omega}\big[S_1>t-\tau_t\big]\\
&=\mu_{\omega,\tau_t}\big(B_{w/\beta,\delta/2}\big)P_{f^{\tau_t}\omega}\big[S_1>t-\tau_t\big].
\label{lower_free_3}
\end{align}
Since $\lim_{t\uparrow\infty}\tau_t/t=\beta$, we have $\liminf_{t\uparrow\infty}\frac{1}{t}\log\mu_{\omega,\tau_t}(B_{w/\beta,\delta/2})\ge-\beta
J_o(w/\beta)$ and
\begin{equation*}
\liminf_{t\uparrow\infty}\frac{1}{t}\log P_{f^{\tau_t}\omega}\big[S_1>t-\tau_t\big]\ge(1-\beta)\ell.
\end{equation*}
The latter limit is trivial for $\beta=1$ because $P_\omega[S_1>0]=1$, whereas for $\beta<1$ it follows from the fact that, for any $\epsilon>0$,
\begin{align}
\nonumber
\liminf_{t\uparrow\infty}\frac{1}{t}\log P_{f^{\tau_t}\omega}\big[S_1>t-\tau_t\big]
&\ge (1-\beta) \liminf_{s\uparrow\infty}\inf_{\tau\in\N_0}\bigg\{\frac{1}{s}\log P_{f^\tau\omega}\big[S_1>s\big]
+\epsilon\frac{\tau}{s}\bigg\}-\epsilon\\
\nonumber
&\ge (1-\beta)\ell-\epsilon.
\end{align}
These arguments, in combination with (\ref{lower_free_3}), show that
\begin{align}
\nonumber
\liminf_{t\uparrow\infty}\frac{1}{t}\log \nu_{\omega,t}(B_{w,\delta})&\ge-\beta J_o(w/\beta)+(1-\beta)\ell\\
\nonumber  
&=-\sup_{\varphi\in\rew^\star}\big\{\varphi(w)-\beta z_o(\varphi)\big\}+(1-\beta)\ell\\
\nonumber
&=-\sup_{\varphi\in\scriptsize{\dom}z}\Big\{\varphi(w)-\beta z(\varphi)-(1-\beta)\ell\Big\},
\end{align}
which is (\ref{lower_free_2}) under the hypothesis $\beta>0$.

In order to settle the case $\beta=0$, we observe that, as $z$ is proper convex and lower semi-continuous, there exist a point $u\in\rew$ and a real number $c$ such that $z(\varphi)\ge\varphi(u)-c$ for every $\varphi\in\rew^\star$ (see \cite{zalinescubook}, Theorem 2.2.6). For all $\epsilon\in(0,1)$ such that $\epsilon\|u\|<\delta/2$ we have $B_{w+\epsilon u,\delta/2}\subseteq B_{w,\delta}$. Thus, the bound (\ref{lower_free_2}) with $\epsilon$ in place of $\beta$ and $w+\epsilon u$ in place of $w$ gives
\begin{align}
\nonumber
\liminf_{t\uparrow\infty}\frac{1}{t}\log \nu_{\omega,t}(B_{w,\delta})
&\ge\liminf_{t\uparrow\infty}\frac{1}{t}\log \nu_{\omega,t}(B_{w+\epsilon u,\delta/2})\\
\nonumber
&\ge-\sup_{\varphi\in \scriptsize{\dom}z}\Big\{\varphi(w+\epsilon u)-\epsilon z(\varphi)-(1-\epsilon)\ell\Big\}\\
\nonumber
&\ge-\sup_{\varphi\in \scriptsize{\dom}z}\big\{\varphi(w)-\ell\big\}-\epsilon(c+\ell).
\end{align}
From here we obtain (\ref{lower_free_2}) corresponding to $\beta=0$ by letting $\epsilon\downarrow 0$. 
\end{proof}

In the case $\ell>-\infty$, our derivation of the large deviation upper bound for compact sets requires that $\rew^\star$ is separable. In fact, due to the difficulty to deal with the measure $E_\omega\big[\mathds{1}_{\{\frac{W_\tau}{t}\in \,\bdot\,,\,\tau\in\mathcal{T}\}} \ee^{H_{\omega,\tau}}\big]$ in (\ref{mu2nu}) when $\tau$ grows slower than $t$, we resort to an estimate based on the cumulant generating function. The latter needs that $\rew^\star$ is not too large in order to ensure uniformity with respect to the disorder. The following lemma shows what the separability of
$\rew^\star$ is needed for.

\begin{lemma}
\label{lem:Legendre_sep}
If $\rew^\star$ is separable, then there exists a countable subset $\mathcal{G}\subset\rew^\star$ such that, for $w\in\rew$,
\begin{equation*}
I_\ell(w)=\sup_{\varphi\in\mathcal{G}}\big\{\varphi(w)-z_\ell(\varphi)\big\}.
\end{equation*}
\end{lemma}

\begin{proof}[Proof of Lemma \ref{lem:Legendre_sep}]
By hypothesis, there exists a countable dense subset $\mathcal{S}^\star$ of $\rew^\star$. Denote by $\mathbb{Q}$ the collection of rational numbers and, for $\vartheta\in\mathcal{S}^\star$ and $\alpha\in\mathbb{Q}$, define
\begin{equation*}
\delta_{\vartheta,\alpha}:=\inf_{\varphi\in\scriptsize{\dom}z_\ell}\max\Big\{\|\vartheta-\varphi\|,|\alpha-z_\ell(\varphi)|\Big\},
\end{equation*}
where $\dom z_\ell:=\{\varphi\in\rew^\star:z_\ell(\varphi)<+\infty\}$. For $n\in\N$, pick $g_{\vartheta,\alpha,n}\in\dom z_\ell$ in such a way that $\max\{\|\vartheta-g_{\vartheta,\alpha,n}\|,|\alpha-z_\ell(g_{\vartheta,\alpha,n})|\}\le \delta_{\vartheta,\alpha}+\frac{1}{n}$. The countable family $\mathcal{G}:=\{g_{\vartheta,\alpha,n}\}_{\vartheta\in\mathcal{S}^\star,\alpha\in\mathbb{Q},n\in\N}$ has the following useful property: for any $\varphi\in\dom z_\ell$ and $\epsilon>0$ there exists a $g\in\mathcal{G}$ such that
\begin{equation*}
\max\{\|\varphi-g\|,|z_\ell(\varphi)-z_\ell(g)|\}<3\epsilon.
\end{equation*}
In fact, given $\varphi\in\dom z_\ell$ and $\epsilon>0$, there exists a $\vartheta\in\mathcal{S}^\star$ such that $\|\varphi-\vartheta\|<\epsilon$ and $\alpha\in\mathbb{Q}$ such that $|z_\ell(\varphi)-\alpha|<\epsilon$. It follows that $\delta_{\vartheta,\alpha}<\epsilon$. Then, for $n\in\N$ satisfying
$\frac{1}{n}<\epsilon$, we find
\begin{align}
\nonumber
\max\Big\{\|\varphi-g_{\vartheta,\alpha,n}\|,|z_\ell(\varphi)-z_\ell(g_{\vartheta,\alpha,n})|\Big\}
&\le \max\Big\{\|\varphi-\vartheta\|,|z_\ell(\varphi)-\alpha)|\Big\}\\
\nonumber
&+\max\Big\{\|\vartheta-g_{\vartheta,\alpha,n}\|,|\alpha-z_\ell(g_{\vartheta,\alpha,n})|\Big\}\\
\nonumber
&\le\max\Big\{\|\varphi-\vartheta\|,|z_\ell(\varphi)-\alpha)|\Big\}+\delta_{\vartheta,\alpha}+\frac{1}{n}<3\epsilon.
\end{align}

The family $\mathcal{G}$ is the desired countable subset of $\rew^\star$. Indeed, for $w\in\rew$,
\begin{equation}
I_\ell(w)=\sup_{\varphi\in\scriptsize{\dom}z_\ell}\big\{\varphi(w)-z_\ell(\varphi)\big\}
\le\sup_{\varphi\in\mathcal{G}}\big\{\varphi(w)-z_\ell(\varphi)\big\},
\label{lem:Legendre_sep_1}
\end{equation}
and hence $I_\ell(w)=\sup_{\varphi\in\mathcal{G}}\{\varphi(w)-z_\ell(\varphi)\}$. To prove (\ref{lem:Legendre_sep_1}), we note that, given $w\in\rew$ and real numbers $\lambda\le I_\ell(w)$ and $\epsilon>0$, there exists a $\varphi\in\dom z_\ell$ such that $\lambda\le \varphi(w)-z_\ell(\varphi)+\epsilon$. Hence there exists a $g\in\mathcal{G}$ such that $\max\{\|\varphi-g\|,|z_\ell(\varphi)-z_\ell(g)|\}<3\epsilon$, as seen above. It follows that
\begin{align}
\nonumber
\lambda\le \varphi(w)-z_\ell(\varphi)+\epsilon 
&\le g(w)-z_\ell(g)+\|\varphi-g\|\|w\|+|z_\ell(\varphi)-z_\ell(g)|+\epsilon\\
\nonumber
&\le g(w)-z_\ell(g)+3\epsilon\|w\|+4\epsilon\\
\nonumber
&\le \sup_{\varphi\in\mathcal{G}}\big\{\varphi(w)-z_\ell(\varphi)\big\}+3\epsilon\|w\|+4\epsilon.
\end{align}
The arbitrariness of $\lambda$ and $\epsilon$ demonstrates (\ref{lem:Legendre_sep_1}).
\end{proof}

We are now in a position to deduce the large deviation upper bound for compact sets, obtaining the following proposition, which is
$(ii)$ of Theorem \ref{WLDPfree} for $\ell>-\infty$.

\begin{proposition}
\label{prop:WLDP_free}
Suppose that $\ell>-\infty$ and $\rew^\star$ is separable. Then $\pae$ the family $\{\nu_{\omega,t}\}_{t \in \N}$ satisfies the weak
LDP with rate function $I_\ell$.
\end{proposition}

\begin{proof}[Proof of Proposition \ref{prop:WLDP_free}] 
In view of Lemma \ref{lem:lower_bound_free}, it remains to verify the quenched large deviation upper bound for compact sets. Let $\mathcal{G}\subset\rew^\star$ be the countable set in Lemma \ref{lem:Legendre_sep}. By Lemma \ref{lem:CGF_free}, there exists a set $\Omega_o\in\mathcal{F}$ with $\prob[\Omega_o]=1$ such that $\lim_{t\uparrow\infty}\frac{1}{t}\log\int_\rew \ee^{tg(v)}\nu_{\omega,t}(\dd v)=z_\ell(g)$ for all $\omega\in\Omega_o$ and $g\in\mathcal{G}$. Pick $\omega\in\Omega_o$. Below we show that
\begin{equation}
\inf_{\delta>0}\bigg\{\limsup_{t\uparrow\infty}\frac{1}{t}\log\nu_{\omega,t}(B_{w,\delta})\bigg\}\le z_\ell(g)-g(w)
\label{prop:WLDP_free_1}
\end{equation}
for all $w\in\rew$ and $g\in\mathcal{G}$. This gives $\inf_{\delta>0}\{\limsup_{t\uparrow\infty}\frac{1}{t}\log\nu_{\omega,t}(B_{w,\delta})\}\le-I_\ell(w)$ for every $w$ by Lemma \ref{lem:Legendre_sep}, which in turn demonstrates the large deviation upper bound for compact sets as in the proof of Proposition \ref{prop:ldp_Jo}.

Given $w\in\rew$, $g\in\mathcal{G}$, and $\delta>0$, the bound $g(v)-g(w)\ge -\|g\|\|v-w\|\ge-\|g\|\delta$ for $v\in B_{w,\delta}$ implies, for $t\in\N$,
\begin{equation*}
\nu_{\omega,t}(B_{w,\delta})= \int_{B_{w,\delta}}\nu_{\omega,t}(\dd v)\le \int_\rew \ee^{tg(v)-tg(w)+t\|g\|\delta}\nu_{\omega,t}(\dd v).
\end{equation*}
Taking logarithm, dividing by $t$, letting first $t\uparrow\infty$ and afterwards $\delta\downarrow 0$, we find (\ref{prop:WLDP_free_1}).
\end{proof}


\subsection{The full LDP}
\label{sec:FLDPfree} 

One of the hypothesis of Corollary \ref{FLDPfree} is that $\rew$ is finite-dimensional. In particular, this implies that $\rew^\star$ is separable, so that $\pae$ the family of measures $\{\nu_{\omega,t}\}_{t\in\N}$ satisfies the weak LDP with rate function $I_\ell$ when either $\ell=-\infty$ or $\ell>-\infty$. Thus, similarly to Corollary \ref{FLDPc}, the full LDP stated by Corollary \ref{FLDPfree} is a consequence of the following quenched
exponential tightness.

\begin{lemma}
\label{lem:exp_tight_free}
Under the hypotheses of Corollary \ref{FLDPc}, $\pae$ the family of measures $\{\nu_{\omega,t}\}_{t\in\N}$ is exponentially tight.
\end{lemma}

\begin{proof}[Proof of Lemma \ref{lem:exp_tight_free}]
Pick a number $\rho>0$, and let $\xi>0$ be as in Corollary \ref{FLDPc}. For $t\in\N$, a Chernoff-type bound, together with the identity
in (\ref{mu2nu}) and the bound in (\ref{H_bound}), gives
\begin{align}
\nonumber 
\nu_{\omega,t}(B_{0,\rho}^c)
&=\sum_{\tau=0}^tE_\omega\bigg[\mathds{1}_{\big\{\frac{W_\tau}{t}\in B_{0,\rho}^c,\,\tau\in\mathcal{T}\big\}} 
\ee^{H_{\omega,\tau}}\bigg] P_{f^\tau\omega}[S_1>t-\tau]\\
\nonumber
&\le \sum_{\tau=0}^tE_\omega\Big[\mathds{1}_{\{\sum_{i=1}^{N_\tau}\|X_i\|\ge \rho t,\,\tau\in\mathcal{T}\}} \ee^{H_{\omega,\tau}}\Big]\\
\nonumber
&\le \ee^{-\xi\rho t}\sum_{\tau=0}^tE_\omega\Big[\mathds{1}_{\{\tau\in\mathcal{T}\}} \ee^{\xi\sum_{i=1}^{N_\tau}\|X_i\|+H_{\omega,\tau}}\Big]\\
&\le\ee^{\sum_{\tau=0}^{t-1}\sup_{s\in\N}\max\{0,v_{f^\tau\omega}(s)-\eta s\}+\eta t-\xi\rho t}
\sum_{\tau=0}^tE_\omega\Big[\mathds{1}_{\{\tau\in\mathcal{T}\}}\ee^{\xi \sum_{i=1}^{N_\tau}\|X_i\|}\Big].
\label{lem:exp_tight_free_1}
\end{align}
On the other hand, with $M$ as in Corollary \ref{FLDPc}, (\ref{lem:exp_tight_2}) implies
\begin{equation}
\sum_{\tau=0}^tE_\omega\Big[\mathds{1}_{\{\tau\in\mathcal{T}\}}\ee^{\xi \sum_{i=1}^{N_\tau}\|X_i\|}\Big]
\le (t+1)\,\ee^{Mt+\sum_{\tau=0}^{t-1}\sup_{s\in\N}\max\big\{0,\log\int_\rew \ee^{\xi\|x\|-Ms}
\lambda_{f^\tau\omega}(\dd x|s)\big\}}.
\label{lem:exp_tight_free_2}
\end{equation}
Combining (\ref{lem:exp_tight_free_1}) with (\ref{lem:exp_tight_free_2}), taking the logarithm, dividing by $t$, and letting $t\uparrow\infty$, we find
\begin{align}
\nonumber
\limsup_{t\uparrow\infty}\frac{1}{t}\log\nu_{\omega,t}(B_{0,\rho}^c)
&\le\Ex\Big[\sup_{s\in\N}\max\big\{0,v_{\bdot}(s)-\eta s\big\}\Big]+\eta-\xi\rho\\
&\qquad +M+\Ex\bigg[\sup_{s\in\N}\max\bigg\{0,\log\int_\rew\ee^{\xi\|x\|-Ms}\lambda_{\bdot}(\dd x|s)\bigg\}\bigg] \qquad \pae
\label{lem:exp_tight_free_3}
\end{align}
thanks to Birkhoff's ergodic theorem. Since $\Ex[\sup_{s\in\N}\max\{0, v_{\bdot}(s)-\eta s\}]<+\infty$ by Assumption \ref{ass3} and $\Ex[\sup_{s\in\N}\max\{0,\log\int_\rew\ee^{\xi\|x\|-Ms}\lambda_{\bdot}(\dd x|s)\}]<+\infty$ by hypothesis, given $\lambda>0$, the lemma follows by choosing $\rho$ so large that the r.h.s.\ of (\ref{lem:exp_tight_free_3}) is smaller than $\lambda$ and $K$ is equal to the closure of $B_{0,\rho}$.
\end{proof}



\end{document}